\documentclass[a4paper,12pt,oneside,reqno]{amsart}

\usepackage[T1]{fontenc}
\usepackage[utf8]{inputenc}
\usepackage{lmodern}
\usepackage[english]{babel}

\usepackage[%
	left=2.5cm,       
	right=2.5cm,      
	top=3.5cm,        
	bottom=3.5cm,     
	heightrounded,    
	bindingoffset=0mm 
]{geometry}

\usepackage{microtype}

\usepackage{amsmath}
\usepackage{amssymb}
\usepackage{mathtools}
\usepackage{mathrsfs}
\usepackage{comment}

\usepackage[
colorlinks=true,
urlcolor=teal,
citecolor=magenta,
linkcolor=teal,
breaklinks=true]
{hyperref} 

\usepackage[noabbrev,capitalize,nameinlink]{cleveref}

\usepackage{bookmark}

\usepackage[initials,
msc-links,
]{amsrefs}

\usepackage{graphicx}
\usepackage{xcolor}
\usepackage{enumitem}
\usepackage{url}

\usepackage[
final
]{showlabels}

\showlabels{bib}

\numberwithin{equation}{section}

\newcommand{\di}{\,\mathrm{d}}
\newcommand{\R}{\mathbb{R}}

\theoremstyle{plain}
\newtheorem{theorem}{Theorem}[section]
\newtheorem{corollary}[theorem]{Corollary}
\newtheorem{proposition}[theorem]{Proposition}
\newtheorem{lemma}[theorem]{Lemma}

\theoremstyle{definition}
\newtheorem{definition}[theorem]{Definition}

\newtheorem{remark}[theorem]{Remark}

\begin{document}
	
\title[A geometrical approach to the sharp fractional Hardy inequality]{A geometrical approach to the sharp Hardy inequality in Sobolev--Slobodecki\u{\i} spaces}

\author[F.~Bianchi]{Francesca Bianchi}
\address[F.~Bianchi]{Dipartimento di Scienze Matematiche, Fisiche e Informatiche, Università degli Studi di Parma, Parco Area delle Scienze 53/A (Campus Scienze e Tecnologie) 43124 Parma (PR), Italy}
\email{francesca.bianchi@unipr.it}

\author[G.~Stefani]{Giorgio Stefani}
\address[G.~Stefani]{Dipartimento di
Matematica ``Tullio Levi-Civita'', Università degli Studi di Padova, Via Trieste 63, 35121 Padova (PD), Italy}
\email{giorgio.stefani@unipd.it}

\author[A.~C.~Zagati]{Anna Chiara Zagati}
\address[A.~C.~Zagati]{Dipartimento di Scienze Matematiche, Fisiche e Informatiche, Università degli Studi di Parma, Parco Area delle Scienze 53/A (Campus Scienze e Tecnologie) 43124 Parma (PR), Italy}
\email{annachiara.zagati@unipr.it}

\date{\today}



\thanks{%
\textit{Acknowledgments}.
The authors thank Lorenzo Brasco for several precious comments on a preliminary version of the present work.
The authors are members of the Istituto Nazionale di Alta Matematica (INdAM), Gruppo Nazionale per l'Analisi Matematica, la Probabilità e le loro Applicazioni (GNAMPA).
F.~Bianchi has received funding from INdAM under the INdAM--GNAMPA Project 2024 \textit{Nuove frontiere nella capillarità non locale} (grant agreement No.\ CUP\_E53\-C23\-001\-670\-001).
G.~Stefani has received funding from INdAM under the INdAM--GNAMPA 2023 Project \textit{Problemi variazionali per funzionali e operatori non-locali} (grant agreement No.\ CUP\_E53\-C22\-001\-930\-001), the INdAM--GNAMPA 2024 Project  \textit{Ottimizzazione e disuguaglianze funzionali per problemi geometrico-spettrali locali e non-locali} (grant agreement No.\ CUP\_E53\-C23\-001\-670\-001) and the INdAM--GNAMPA Project 2025 \textit{Metodi
variazionali per problemi dipendenti da operatori frazionari isotropi e
anisotropi} (grant agreement No.\ CUP\_E53\-240\-019\-500\-01), from the European Research Council (ERC) under the European Union’s Horizon 2020 research and innovation program (grant agreement No.~945655), and from the
European Union -- NextGenerationEU and the University of Padua under the 2023
STARS@UNIPD  Starting Grant Project \textit{New Directions in Fractional
Calculus -- NewFrac} (grant agreement No.\ CUP\_C95\-F21\-009\-990\-001).
A.~C.~Zagati has received funding from INdAM under the INdAM--GNAMPA Project 2024 \textit{Proprietà geometriche e regolarità in problemi variazionali locali e non locali} (grant agreement No.\ CUP\_E53\-C23\-001\-670\-001).%
}

\begin{abstract}
We give a partial negative answer to a question left open in a previous work by Brasco and the first and third-named authors concerning the sharp constant in the fractional Hardy inequality on convex sets.
Our approach has a  geometrical flavor and equivalently reformulates the sharp constant in the limit case $p=1$ as the Cheeger constant for the fractional perimeter and the Lebesgue measure with a suitable weight.
As a by-product, we obtain new lower bounds on the sharp constant in the $1$-dimensional case, even for non-convex sets, some of which optimal in the case $p=1$.
\end{abstract}

\maketitle	
	

\section{Introduction}

\subsection{Classical framework}

Given $N\in\mathbb N$ and $p\in(1,\infty)$, the \emph{Hardy inequality} on a (non-empty) open convex set $\Omega\subsetneq\R^N$ states that there exists $C>0$ such that 
\begin{equation}
\label{eq:inthardy}
\int_{\R^N}|\nabla u|^p\di x
\ge
C\int_\Omega\frac{|u|^p}{d_\Omega^p}\di x,
\quad
\text{for all}\
u\in C^\infty_0(\Omega),
\end{equation}
see~\cites{OK90,M11} for a detailed introduction.
Here and in the following, given a (non-empty) open set $\Omega\subseteq\R^N$, we let $C^\infty_0(\Omega)$ be the set of all smooth functions on $\R^N$ with compact support contained in $\Omega$. 
In addition, provided that $\Omega\ne\R^N$, we let $d_\Omega\colon\Omega\to[0,\infty)$,
\begin{equation}
\label{eq:dist_func}
d_\Omega(x)
=
\inf\limits_{y\in\partial\Omega}|x-y|,
\quad
\text{for all}\ 
x\in\Omega,
\end{equation}
be the \emph{distance} function from $\partial\Omega$.

The sharp constant in~\eqref{eq:inthardy}, defined as
\begin{equation}
\label{eq:intvarp}
\mathfrak h_{1,p}(\Omega)
=
\inf_{u\in C^\infty_0(\Omega)}
\left\{
\int_{\R^N}|\nabla u|^p\di x
:
\int_{\Omega}\frac{|u|^p}{d_\Omega^p}\di x=1
\right\}\in[0,\infty),
\end{equation}
can be explicitly computed as
\begin{equation}
\label{eq:intsharp}
\mathfrak h_{1,p}(\Omega)
=
\mathfrak h_{1,p}(\mathbb H^N_+)
=
\left(\frac{p-1}{p}\right)^p,
\end{equation}  
see the original Hardy's works~\cites{H20,H25} for $N=1$ and~\cites{MMP98,MS97} for $N\ge2$.
As well-known, inequality~\eqref{eq:inthardy} cannot hold for $p=1$; that is, 
\begin{equation}
\label{eq:inthardy_no1}
\mathfrak h_{1,1}(\Omega)=\mathfrak h_{1,1}(\mathbb H^N_+)=0
\end{equation}
in~\eqref{eq:intvarp}.
Here and in the following, we set
\begin{equation*}
\mathbb H^1_+=(0,\infty)
\quad
\text{and}
\quad
\mathbb H^N_+=\R^{N-1}\times(0,\infty)
\quad
\text{for}\
N\ge2.
\end{equation*}

Noteworthy, the sharp constant~\eqref{eq:intsharp} is never attained in~\eqref{eq:inthardy}, not even in the corresponding homogeneous Sobolev space $\mathcal W^{1,p}_0(\Omega)$ obtained as the completion of $C^\infty_0(\R^N)$ with respect to the left-hand side of~\eqref{eq:inthardy}. 

\subsection{Fractional framework}

The fractional analog of~\eqref{eq:inthardy} was achieved in~\cite{BC18}*{Th.~1.1}.
Precisely, given $s\in(0,1)$ and $p\in(1,\infty)$, there exists $C>0$ such that 
\begin{equation}
\label{eq:frachardy}
[u]^p_{W^{s,p}(\R^N)}
\ge 
C\int_\Omega\frac{|u|^p}{d_\Omega^{sp}}\di x,
\quad
\text{for all}\
u\in C^\infty_0(\Omega).
\end{equation}
Here and in the rest of the paper, given $s\in(0,1)$, $p\in[1,\infty)$, an open set $A\subset\R^N$, and a measurable function $u\colon A\to[-\infty,\infty]$, we let 
\begin{equation}
\label{eq:frac_seminorm}
[u]_{W^{s,p}(A)}
=
\left(
\int_{A}
\int_{A}
\frac{|u(x)-u(y)|^p}{|x-y|^{N+sp}}
\di x
\di y
\right)^{\frac1p}
\end{equation}
be the \emph{Sobolev--Slobodecki\u{\i} $(s,p)$-fractional seminorm} of~$u$ on~$A$.
The sharp constant in~\eqref{eq:frachardy}, defined analogously as in~\eqref{eq:intvarp} as
\begin{equation}\label{eq:h_sp}
\mathfrak{h}_{s,p}(\Omega)
=
\inf_{u \in C^{\infty}_0(\Omega)} 
\left\{
[u]_{W^{s,p}(\mathbb{R}^N)}^p : \int_{\Omega} \frac{|u|^p}{d_{\Omega}^{sp}} \di x = 1 
\right\}\in[0,\infty),
\end{equation}
was investigated by L.~Brasco and the first and third-named authors in the recent work~\cite{BBZ23}.
In~\cite{BBZ23}*{Main Th.}, it was proved that, for all $p\in(1,\infty)$ and $s\in(0,1)$,
\begin{equation}
\label{eq:sharpfrac_h}
\mathfrak h_{s,p}(\mathbb H^N_+)
=
C_{N,sp}\,\Lambda_{s,p}
\end{equation}
and, in analogy with the classical result~\eqref{eq:intsharp}, if either $p=2$ or $
sp\ge 1$, then
\begin{equation}
\label{eq:sharpfrac}
\mathfrak h_{s,p}(\Omega)=\mathfrak h_{s,p}(\mathbb H^N_+). 
\end{equation}
Moreover, as in the classical case, the sharp constant in~\eqref{eq:sharpfrac} is never attained in~\eqref{eq:frachardy}, not even in the homogeneous fractional Sobolev--Slobodecki\u{\i} space $\mathcal W^{s,p}_0(\Omega)$ obtained as the completion of $C^\infty_0(\Omega)$ with respect to the seminorm~\eqref{eq:frac_seminorm} with $A=\R^N$. 
Here and below, we set
\begin{equation}
\label{eq:C_Nsp}
C_{N,q}=
\begin{cases}
(N-1)\,\omega_{N-1}\displaystyle\int_0^{\infty} \frac{t^{N-2}}{(1+t^2)^{\frac{N+q}{2}}}
\di t
&
\quad
\text{for}\ N\ge2,
\\[1.5ex]
1
&
\quad
\text{for}\ N=1,
\end{cases}
\end{equation}
whenever $q\in[0,\infty)$, and
\begin{equation}
\label{eq:Lambda_sp}
\Lambda_{s,p}
=
2\int_0^1 \frac{\big|1-t^{\frac{sp-1}{p}}\big|^p}{(1-t)^{1+sp}}
\di t
+
\dfrac{2}{sp}
\end{equation}
whenever $s\in(0,1)$ and $p\in[1,\infty)$.

As mentioned in~\cite{BBZ23}, concerning~\eqref{eq:sharpfrac}, the cases $p=2$ for $\Omega=\mathbb H^N_+$, and $p=2$ with $s\ge\frac12$ for any open convex set $\Omega\subsetneq\R^N$, were previously established  respectively in~\cite{BD11}*{Th.~1.1} and in~\cite{FMT13}*{Th.~5} (see also~\cite{FMT18}), with different techniques.

The strategy of~\cite{BBZ23} for proving~\eqref{eq:sharpfrac_h} and~\eqref{eq:sharpfrac} expands on the ideas of~\cite{BD11} and relies on an equivalent characterization of the constant~\eqref{eq:h_sp} via the existence of positive local weak \emph{supersolutions} of the corresponding non-local eigenvalue problem, see~\cite{BBZ23}*{Eq.~(1.6)}.

Unfortunately, this approach does not seem to work for determining the sharp constant $\mathfrak h_{s,p}(\Omega)$ for $\Omega\ne\mathbb H^N_+$ for the values of $s$ and $p$ not included in~\eqref{eq:sharpfrac}, see the discussion in~\cite{BBZ23}*{Sec.~1.3}.
Nevertheless, in virtue of~\cite{BBZ23}*{Rem.~6.4}, it holds that
\begin{equation}
\label{eq:frachardy_estim}
\frac{2\,C_{N,sp}}{sp}
\le 
\mathfrak h_{s,p}(\Omega)
\le 
\mathfrak h_{s,p}(\mathbb H^N_+)
\end{equation}
for all $p\in(1,\infty)$ and $s\in(0,1)$.

\subsection{The convex case}

Our main aim is to tackle the question left open in~\cite{BBZ23}*{Open Prob.} concerning the validity of~\eqref{eq:sharpfrac} for $p\ne2$ and $s<\frac1p$.

We first observe that a plain limiting argument yields~\eqref{eq:sharpfrac_h} and~\eqref{eq:frachardy_estim}, and thus~\eqref{eq:frachardy}, for $p=1$---remarkably in contrast with~\eqref{eq:inthardy}, recall~\eqref{eq:inthardy_no1}.
We thus complete the picture of~\cites{BC18,BBZ23} with the case $p=1$, with the characterization of~\eqref{eq:h_sp} for the half-space.
Here and below, given $N\ge1$, $s\in(0,1)$ and a non-empty open set $\Omega\subsetneq\R^N$, we let 
\begin{equation}
\label{eq:h_s1}
\mathfrak{h}_{s,1}(\Omega)
=
\inf_{u \in C^{\infty}_0(\Omega)} 
\left\{
[u]_{W^{s,1}(\mathbb{R}^N)} : \int_{\Omega} \frac{|u|}{d_{\Omega}^{s}} \di x = 1 
\right\}\in[0,\infty).
\end{equation}
Arguing as in~\cite{BBZ23}*{Rem.~3.1}, $C^\infty_0(\Omega)$ can be replaced with $\mathcal W^{s,1}_0(\Omega)$  in the infimum in~\eqref{eq:h_s1}.

\begin{theorem}
\label{res:hs1}
For $N\ge1$ and $s\in(0,1)$, it holds that
\begin{equation}
\label{eq:hs1_H}
\mathfrak h_{s,1}(\mathbb H^N_+)
=
\Lambda_{s,1}\,C_{N,s}
\quad
\text{with}
\quad
\Lambda_{s,1}=\frac4s,
\end{equation}
and, whenever is $\Omega\subsetneq\R^N$ is a (non-empty) open convex set,
\begin{equation}
\label{eq:hs1}
\frac12\,\mathfrak h_{s,1}(\mathbb H^N_+)
\le 
\mathfrak h_{s,1}(\Omega)
\le 
\mathfrak h_{s,1}(\mathbb H^N_+).
\end{equation}
Moreover, the constant $\mathfrak h_{s,1}(\mathbb H^1_+)$ is attained in $\mathcal W^{s,1}_0(\mathbb H^1_+)$.
\end{theorem}

Having~\eqref{eq:frachardy} for $p=1$ at disposal, our first main result yields a partial negative answer to~\cite{BBZ23}*{Open Prob.} for $p=1$ and $\Omega\subsetneq\R^N$ a (non-empty) bounded open convex set.

\begin{theorem}
\label{res:homo}
Let $N\ge1$ and let $\Omega\subsetneq\R^N$ be a (non-empty) bounded open convex set.
If $N=1$, then 
\begin{equation}
\label{eq:homo_1}
\mathfrak h_{s,1}(\Omega)=2^{-s}\,\mathfrak h_{s,1}(\mathbb H^1_+)
=
\frac{2^{2-s}}{s}
\quad
\text{for all}\
s\in(0,1)
\end{equation}
and the constant $\mathfrak h_{s,1}(\Omega)$ is attained in $\mathcal W^{s,1}_0(\Omega)$.
If $N\ge2$, then
\begin{equation}
\label{eq:homo_lim}
\lim_{s\to1^-}
\frac{\mathfrak h_{s,1}(\Omega)}{\mathfrak h_{s,1}(\mathbb H^N_+)}
=
\frac{1}{2}.
\end{equation} 
\end{theorem}

Noteworthy, for $N\ge2$, in the special case $\Omega$ is an open ball, the limit~\eqref{eq:homo_lim} can be slightly refined by exploiting the main result of~\cite{G20}, see \cref{res:ball} below.

\cref{res:hs1,res:homo} uncover some remarkable differences between the classical inequality~\eqref{eq:inthardy} and its fractional counterpart~\eqref{eq:frachardy}, revealing some unexpected features of the sharp constant~\eqref{eq:h_s1}.
In particular, the first inequality in~\eqref{eq:hs1} is asymptotically optimal as $s\to1^-$ for (non-empty) bounded open convex sets.
Moreover, for $N=1$, the sharp constant~\eqref{eq:h_s1} is attained, both in the bounded and unbounded case.
We do not know if this might be the case also for $N\ge2$, but we leave this interesting question for future investigations. 
For further discussions, see also \cref{rem:limits} below.

As a consequence of \cref{res:homo}, we can give the following partial negative answer to~\cite{BBZ23}*{Open Prob.} for $p>1$, suggesting that the first inequality in~\eqref{eq:frachardy_estim} is asymptotically optimal as $s\to1^-$ and $p\to1^+$, with $sp<1$, for (non-empty) bounded  open convex sets.

\begin{corollary}
\label{res:salto}
If $N\ge1$ and  $\Omega\subsetneq\R^N$ is a (non-empty) bounded open convex set, then
\begin{equation*}
\lim_{s\to1^-}
\limsup_{p\to1^+}
\frac{\mathfrak h_{s,p}(\Omega)}{\mathfrak h_{s,p}(\mathbb H^N_+)}
=
\frac12.
\end{equation*}
\end{corollary}

\subsection{Geometrical approach}

Besides refining and extending some of the results of~\cite{BBZ23} to the limit case $p=1$, the key idea of our approach is to reinterpret the sharp constant $\mathfrak h_{s,1}(\Omega)$ in~\eqref{eq:h_sp} in a geometrical sense. 
More precisely, we prove the following result (in which we do not need to assume that $\Omega$ is convex).

\begin{theorem}
\label{res:cheeger}
Let $N\ge1$, $s\in(0,1)$ and let $\Omega\subsetneq\R^N$ be a (non-empty)  open set.
Letting 
\begin{equation*}
P_s(E)
=
[\mathbf{1}_E]_{W^{s,1}(\mathbb{R}^N)}
\quad
\text{and}
\quad
V_{s,\Omega}(E)
=
\int_{E}\frac{1}{d_{\Omega}^{s}} \di x
\end{equation*}
for any measurable set $E\subseteq\Omega$, it holds that
\begin{equation}
\label{eq:cheeger_eq}
\mathfrak{h}_{s,1}(\Omega)=\mathfrak{g}_{s}(\Omega),
\end{equation}
where
\begin{equation}
\label{eq:cheeger}
\mathfrak{g}_{s}(\Omega) 
= 
\inf \left\{ \frac{P_s(E)}{V_{s,\Omega}(E)} : E \Subset\Omega,\ |E|>0 \right\}\in[0,\infty).
\end{equation}
\end{theorem}

\cref{res:cheeger} revisits the well-known connection between the \emph{Cheeger constant} of a set with the \emph{first $1$-eigenvalue} of the underlying \emph{$1$-Laplacian} (see~\cite{FPSS24} and the references therein for a detailed account) in terms of the fractional Hardy inequality~\eqref{eq:frachardy} for $p=1$. 
In fact, the problem in~\eqref{eq:cheeger} is the \emph{Cheeger problem} in $\Omega$ with respect to the \emph{fractional perimeter}~$P_s$ and the \emph{weighted} volume~$V_{s,\Omega}$, so that \cref{res:cheeger} can be seen as a particular case of~\cite{FPSS24}*{Th.~5.4} (see also~\cite{BLP14}*{Th.~5.8} and~\cite{BS24}*{Th.~3.10} for similar results in the fractional/non-local setting).
The idea behind the proof of \cref{res:cheeger} is essentially to exploit the \emph{fractional coarea formula} (first observed in~\cite{V91})
\begin{equation}
\label{eq:fracoarea}
[u]_{W^{s,1}(\R^N)}
=
\int_{-\infty}^\infty P_s(\{u>t\})\di t
\end{equation}
in order to pass from a function $u\in C^\infty_0(\Omega)$ in~\eqref{eq:h_sp} to its superlevel sets $\{u>t\}$ in~\eqref{eq:cheeger}.

With \cref{res:cheeger} at disposal, \cref{res:homo} basically follows from the fact that
\begin{equation*}
\mathfrak h_{s,1}(\Omega)
=
\mathfrak g_s(\Omega)
\le
\frac{P_s(\Omega)}{V_{s,\Omega}(\Omega)}. 
\end{equation*}
For $N\ge2$, the idea is first to rewrite $V_{s,\Omega}(\Omega)$ in terms of the superlevel sets of $d_\Omega$ via the usual coarea formula and then to take advantage of the well-known limit (e.g., see~\cites{D02,L19})
\begin{equation}
\label{eq:davila}
\lim_{s\to1^-}
(1-s)\,P_s(\Omega)
=
2\,\omega_{N-1}\,P(\Omega),
\end{equation}
valid for any measurable set $\Omega\subsetneq\R^N$ such that $\mathbf 1_\Omega\in BV(\R^N)$.
The case $N=1$ in~\eqref{eq:homo_1}, instead, relies on a rearrangement argument (see \cref{res:segment} below) which, unfortunately, does not work in higher dimensions (for more details, see \cref{rem:lucky_1d}).

\subsection{The non-convex case}

The principle underlying \cref{res:cheeger} can be also exploited to deal with the non-convex case $\Omega=\R^N\setminus\{0\}$.
In fact, our geometrical approach allows us to naturally recover~\cite{FS08}*{Th.~1.1} for $p=1$ (compare also with the argument in~\cite{FS08}*{Sec.~3.4}). 
In more precise terms, we can (re)prove the following result.

\begin{corollary}
\label{res:fs}
Given $N\ge1$ and $s\in(0,1)$, it holds that
\begin{equation}
\label{eq:fs_h}
[u]_{W^{s,1}(\R^N)}
\ge
\mathfrak h_{s,1}(\R^N\setminus\{0\})
\,
\int_{\R^N}\frac{|u(x)|}{|x|^s}\di x
\end{equation}
for all $u\in\mathcal W^{s,1}_0(\R^N)$,
with sharp constant given by
\begin{equation}
\label{eq:fs_c}
\mathfrak h_{s,1}(\R^N\setminus\{0\})
=
\frac{4}{s}
\,
\frac{\pi^{\frac N2}\,\Gamma(1-s)}{\Gamma\left(\frac{N-s}{2}\right)\,\Gamma\left(1-\frac s2\right)}.
\end{equation}
Equality holds in~\eqref{eq:fs_h} if and only if $u$ is proportional to a symmetric decreasing function. 
\end{corollary}

Noteworthy, for $N=1$, the constant in~\eqref{eq:fs_c} equals the one in~\eqref{eq:homo_1} for all $s\in(0,1)$,
\begin{equation}
\label{eq:fs1}
\mathfrak h_{s,1}\left(\R\setminus\{0\}\right)
=
\frac{2^{2-s}}{s},
\end{equation}
by well-known properties of the Gamma function (e.g., the \emph{Legendre duplication formula}).
The equality in~\eqref{eq:fs1} is actually a particular case of the following result (including~\eqref{eq:homo_1} as a special instance), which comes as a natural by-product of our geometrical approach (as customary, $\lfloor x\rfloor\in\mathbb Z$ denotes the largest integer smaller than $x\in\R$). 

\begin{theorem}
\label{res:onen}
Let $N=1$ and $\Omega\subsetneq\R$ be the union of $n\in\mathbb N$ (non-empty) disjoint bounded open intervals  $\{I_k:k=1,\dots,n\}$.
For $s\in(0,1)$ and $p\in[1,\infty)$ such that $sp<1$, it holds
\begin{equation}
\label{eq:onen}
\mathfrak h_{s,p}(\Omega)
\ge 
\frac{\mathfrak h_{s,p}(\R\setminus\{0\})}{n^{sp}}.
\end{equation}
Moreover, if $p=1$ and $\Omega$ is equivalent to an interval, then~\eqref{eq:onen} holds as an equality if and only if all the $I_k$'s have the same measure. 
On the other hand, it holds that
\begin{equation}
\label{eq:rmenz}
\mathfrak h_{s,p}((-R,R)\setminus\mathbb Z)
\le 
\frac{R^{1-sp}}{\lfloor R\rfloor}
\,
\frac{4^{1-sp}}{sp}
\end{equation}
for any $R\ge 1$, with $\mathfrak h_{s,p}(\R\setminus\mathbb Z)=0$ in the limit case $R=\infty$.
\end{theorem} 

We emphasize that the intervals in the first part of the statement of \cref{res:onen} do not have to be at a positive distance apart.
If, instead, one wants to keep track of the (minimal) distance between the intervals, then an alternative route is possible.
This is the content of our last main result (see also \cref{res:nsegments_delta} below).

\begin{theorem}
\label{res:nsegments}
Let $\Omega\subsetneq\R$ be the union of $n\in\mathbb N$ non-empty bounded open intervals, each with measure at most $\ell\in(0,\infty)$ and at distance at least $\delta\in(0,\infty)$ from the others.
For $s\in(0,1)$ and $p\in[1,\infty)$ such that $sp\le1$, it holds that
\begin{equation*}
\mathfrak h_{s,p}(\Omega)
\ge
\frac{2^{2-sp}}{sp}
\left(
1-
\left(\frac{\ell}{\ell+\delta}\right)^{sp}
\,
\right).
\end{equation*}
\end{theorem}

It is worth observing that, in the case $sp<1$, due to \cref{res:spacco} below, the above inequality implies that
\begin{equation*}
\mathfrak h_{s,p}(\Omega)
\ge
\mathfrak h_{s,p}(\R\setminus\{0\})
\left(
1-
\left(\frac{\ell}{\ell+\delta}\right)^{sp}
\,
\right),
\end{equation*}
which can be more easily paired with~\eqref{eq:onen}.

We point out that inequality~\eqref{eq:frachardy}  was addressed in the non-convex case in~\cites{S23,CP24}.
The validity of~\eqref{eq:frachardy} for any open set $\Omega\subsetneq\R^N$  was achieved in~\cite{S23}*{Th.~1.10} with $s\in(0,1)$ and $p\in(1,\infty)$ such that $sp>N>1$, but with an implicit constant.
Recently, this result has been improved in~\cite{CP24}*{Th.~1.1}, where it was shown that
\begin{equation}
\label{eq:cintiprinari}
\mathfrak h_{s,p}(\Omega)
\ge
\mathfrak h_{s,p}(\R^N\setminus\{0\}) 
\end{equation} 
for any open set $\Omega\subsetneq\R^N$ with $s\in(0,1)$ and $p\in(1,\infty)$ such that $sp>N\ge1$.   
To the best of our knowledge, \cref{res:onen,res:nsegments} are the first results establishing inequality~\eqref{eq:frachardy} on a (non-empty) open non-convex set $\Omega\subsetneq\R$ for $s\in(0,1)$ and $p\in[1,\infty)$ such that $sp\le N=1$.
Furthermore,  \cref{res:onen} shows that inequality~\eqref{eq:cintiprinari} cannot hold in the case $p=1$ and $N=1$ for any $s\in(0,1)$.

\begin{remark}[Regional fractional Hardy inequality]
\label{rem:regional}
For completeness, we mention that a vast literature has been devoted to the study of the \emph{regional} version of inequality~\eqref{eq:frachardy},
\begin{equation}
\label{eq:frac_regional}
[u]^p_{W^{s,p}(\Omega)}
\ge 
C
\int_\Omega\frac{|u|^p}{d_\Omega^{sp}}\di x,
\quad
\text{for all}\
u\in C^\infty_0(\Omega).
\end{equation} 
Far from being complete, we refer  to~\cites{D04,DK22,DV14,DV15,EHV14,S23} and to the references therein for sufficient and/or necessary conditions on $s$, $p$ and $\Omega$ for~\eqref{eq:frac_regional} to hold.
The problem of determining the sharp constant in~\eqref{eq:frac_regional} has been addressed in~\cites{BD11,FS10,LS10}.
In contrast to~\eqref{eq:frachardy}, in this case the restriction $sp>1$ is needed, since~\eqref{eq:frac_regional} cannot hold for $sp\le1$ on bounded open Lipschitz sets $\Omega\subsetneq\R^N$, see~\cite{D04}*{Sec.~2}.
For further discussions, refer to~\cite{BBZ23}*{Rem.~1.1}.
\end{remark} 

\section{Preliminaries}

In this section, we collect some useful notation and preliminary results.

\subsection{Gamma and Beta functions}

For $a,b>0$, we let 
\begin{equation}
\label{eq:gamma_beta}
\Gamma(a)
=
\int_0^\infty t^{a-1}\,e^{-t}\di t
\quad
\text{and}
\quad
\mathrm B(a,b)
=
\int_0^1 t^{a-1}\,(1-t)^{b-1}\di t
\end{equation}
be the \emph{Gamma} and \emph{Beta functions}, respectively.
We recall that 
\begin{equation}
\label{eq:betag}
\mathrm B(a,b)
=
\frac{\Gamma(a)\,\Gamma(b)}{\Gamma(a+b)}
\quad
\text{for all}\
a,b>0.
\end{equation} 
In addition, for $x\in[0,1)$ and $a>0$ and $b\in\R$, we let
\begin{equation}
\label{eq:ibeta}
\mathrm B(x;a,b)
=
\int_0^x t^{a-1}\,(1-t)^{b-1}\di t
\end{equation} 
be the \emph{incomplete Beta function}. 
We recall that, given $a>0$,
\begin{equation}
\label{eq:ibeta_sym}
\mathrm B(x;a,-a)
=
\frac{x^a}{a\,(1-x)^a}
\quad
\text{for all}\
x\in[0,1).
\end{equation}
Formula~\eqref{eq:ibeta_sym} plainly follows from the representation
\begin{equation*}
\mathrm B(x;a,b)
=
\frac{x^a}{a}
\,
F(a,1-b;a+1;x),
\quad
\text{for all}\
x\in[0,1),
\end{equation*}
and the well-known properties of the \emph{hypergeometric function} $F$.

For each $k\in\mathbb N$, we let 
\begin{equation}
\label{eq:ballo}
\omega_k=\frac{\pi^{\frac{k-1}{2}}}{\Gamma\left(\frac{k-1}{2}+1\right)}
\end{equation}
be the $k$-dimensional Lebesgue measure of the unitary ball 
\begin{equation*}
\mathbb B^k
=
\{x\in\R^k : |x|<1\}
\end{equation*}
in $\R^k$.
Recall that $\mathscr H^{k-1}(\partial\mathbb B^k)=k\,\omega_k$, where $\mathscr H^k$ is the $k$-dimensional Hausdorff measure.

\subsection{Symmetric decreasing rearrangement}

Given a measurable set $E\subset\R^N$ such that $|E|\in[0,\infty)$, we let 
\begin{equation}
\label{eq:ballification}
E^\star
=
\left\{x\in\R^N : |x|<r_E\right\}
\end{equation}
where $r_E\in[0,\infty)$ is such that $|E^\star|=|E|$, and we set $\mathbf 1_E^\star=\mathbf 1_{E^\star}$.
The \emph{symmetric decreasing rearrangement} of a measurable function $u\colon\R^N\to\R$ \emph{vanishing at infinity}, i.e., such that 
\begin{equation*}
\left|\left\{x\in\R^N : |u|>t\right\}\right|<\infty
\quad
\text{for all}\
t>0,
\end{equation*}
is defined as 
\begin{equation}
\label{eq:sdr}
u^\star(x)
=
\int_0^\infty
\mathbf 1_{\{|u|>t\}}^\star(x)\di t
\quad
\text{for all}\
x\in\R^N.
\end{equation}
In particular, $u^\star$ is a lower semicountinuous, non-negative, radially symmetric, and non-in\-crea\-sing function such that 
$\{u^\star>t\}=\{|u|>t\}^\star$.
Moreover, if $\Phi\colon[0,\infty)\to[0,\infty)$ is a non-decreasing function, then
\begin{equation}
\label{eq:circostar}
(\Phi\circ|u|)^\star
=
\Phi\circ u^\star
\end{equation}
for any measurable function $u\colon\R^N\to\R$.
In particular, $(|u|^p)^\star=(u^\star)^p$ for all $p\in[1,\infty)$.

The \emph{Hardy--Littlewood inequality} states that, if $u,v\colon\R^N\to\R$ are two measurable functions vanishing at infinity, then 
\begin{equation}
\label{eq:hl}
\int_{\R^N}u(x)\,v(x)\di x
\le 
\int_{\R^N} u^\star(x)\,v^\star(x)\di x.
\end{equation}
For a detailed presentation of~\eqref{eq:sdr} and~\eqref{eq:hl}, we refer to~\cite{LL01}*{Ch.~3} for instance.

\subsection{Fractional setting}

For $s\in(0,1)$ and $p\in[1,\infty)$, we let $\mathcal W_0^{s,p}(\Omega)$ be the completion of $C^\infty_0(\Omega)$ with respect to the seminorm in~\eqref{eq:frac_seminorm} with $A=\R^N$.
Thanks to~\cite{BGV21}*{Th.~3.1},
in the case $sp<N$ (which is the only one we actually need), we can identify
\begin{equation*}
\mathcal W_0^{s,p}(\R^N)
=
\left\{
u\in L^{\frac{Np}{N-sp}}(\R^N)
:
[u]_{W^{s,p}(\R^N)}<\infty
\right\}.
\end{equation*}  
By~\cite{AL89}*{Th.~9.2}, if $u\in\mathcal W^{s,p}_0(\R^N)$, then the symmetric decreasing rearrangement~$u^\star$ of~$u$ defined in~\eqref{eq:sdr} is well defined and such that $u^\star\in \mathcal W^{s,p}_0(\R^N)$, with
\begin{equation} 
\label{eq:posze}
[u^\star]_{W^{s,p}(\R^N)}
\le 
[u]_{W^{s,p}(\R^N)}.
\end{equation}

As customary, given $s\in(0,1)$ and a measurable set $E\subseteq\R^N$, we let 
\begin{equation*}
P_s(E)=[\mathbf 1_E]_{W^{s,1}(\R^N)}
\end{equation*}  
be the \emph{fractional perimeter} of $E$.
We recall that, given two measurable sets $E,F\subseteq\R^N$, the fractional perimeter satisfies \begin{equation}
\label{eq:submodularity}
P_s(E\cap F)
+
P_s(E\cup F)
\le 
P_s(E)
+
P_s(F),
\end{equation}
see~\cite{BS24}*{Sec.~2.5} and the references therein.
Thus, since $P_s(F^c)=P_s(F)$, we also have that
\begin{equation}
\label{eq:perdiff}
P_s(E\setminus F)
=
P_s(E\cap F^c)
\le 
P_s(E)
+
P_s(F^c)
=
P_s(E)
+
P_s(F).
\end{equation}
Moreover, we recall that, if $|E|<\infty$ and $F\subseteq\R^N$ is convex, then
\begin{equation}
\label{eq:interconvex}
P_s(E\cap F)\le P_s(E),
\end{equation} 
see~\cite{BS24}*{Th.~2.29} for instance.

The \emph{fractional isoperimetric inequality} (e.g., see~\cite{G20} and the references therein) states that, for all $s\in(0,1)$, if $E\subseteq\R^N$  is such that $|E|\in(0,\infty)$, then
\begin{equation}
\label{eq:fraciso}
\frac{P_s(E)}{|E|^{1-\frac sN}}
\ge
\frac{P_s(\mathbb B^N)}{|\mathbb B^N|^{1-\frac sN}}. 
\end{equation}
We recall that, thanks to~\cite{G20}*{Prop.~1.1}, we have that
\begin{equation}
\label{eq:garofalo}
P_s(\mathbb B^N)
=
\omega_N^{1-\frac sN}
\,
\frac{N\,\pi^{\frac{N+s}2}\,\Gamma(1-s)}{\frac s2\,\Gamma\left(\frac N2+1\right)^{\frac{s}N}\,\Gamma\left(1-\frac s2\right)\,\Gamma\left(\frac{N-s}{2}+1\right)} 
\end{equation}
for all $s\in(0,1)$.
In particular, for $N=1$, we get that
\begin{equation}
\label{eq:persseg}
P_s((a,b))
=
\frac{4}{s(1-s)}
\,
(b-a)^{1-s}
\end{equation}
whenever $a,b\in\R$ are such that $a<b$.

\subsection{Properties of \texorpdfstring{$C_{N,q}$}{C-N,q} and \texorpdfstring{$\Lambda_{s,p}$}{Lambda-s,p}}

Concerning the constant in~\eqref{eq:C_Nsp}, the following result basically reformulates~\cite{BBZ23}*{Lem.~B.1} and~\cite{FS10}*{Lem.~2.4} with minor refinements. 

\begin{lemma}
\label{res:bambi}
If $N\ge1$ and $q\in[0,\infty)$, then the constant in~\eqref{eq:C_Nsp} rewrites as
\begin{equation}
\label{eq:bambi}
C_{N,q}
=
\frac12\int_{\mathbb S^{N-1}}|x_N|^q\di\mathscr H^{N-1}(x)
=
\pi^{\frac{N-1}2}
\,
\frac{\Gamma\left(\frac{q+1}{2}\right)}{\Gamma\left(\frac{N+q}{2}\right)}.
\end{equation}
In particular, for each $N\ge1$, the function $q\mapsto C_{N,q}$ is continuous for $q\in[0,\infty)$.
\end{lemma}

\begin{proof}
The case $N=1$ is trivial, so we assume $N\ge2$.

The first equality in~\eqref{eq:bambi} has been already proved in~\cite{BBZ23}*{Lem.~B.1}: one just needs to follow the very same argument letting $sp=q$.
Note that the proof of~\cite{BBZ23}*{Lem.~B.1} also works for $q=sp\in[0,1]$.
We leave the simple check to the reader.

For the second inequality in~\eqref{eq:bambi}, it is enough to observe that simple changes of variables yield the formulas (recall the definition in~\eqref{eq:gamma_beta})
\begin{equation*}
\mathrm B(a,b)
=
\int_0^\infty \frac{t^{a-1}}{(1+t)^{a+b}}\di t
=
2\int_0^\infty\frac{t^{2a-1}}{(1+t^2)^{a+b}}\di t,
\quad
\text{for}\ a,b>0.
\end{equation*}
Thus, for $a=\frac{N-1}{2}$ and $b=\frac{q+1}{2}$, and owing to~\eqref{eq:betag}, we can rewrite 
\begin{equation*}
\int_0^\infty\frac{t^{N-2}}{(1+t^2)^{\frac{N+q}{2}}}\di t
=
\frac12\,\mathrm B\left(\frac{N-1}{2},\frac{q+1}2\right)
=
\frac{\Gamma\left(\frac{N-1}2\right)\,\Gamma\left(\frac{q+1}{2}\right)}{2\,\Gamma\left(\frac{N+q}{2}\right)}
.
\end{equation*}
Hence, recalling~\eqref{eq:ballo}, we easily get that 
\begin{equation*}
C_{N,q}
=
(N-1)\,\omega_{N-1}
\,
\frac{\Gamma\left(\frac{N-1}2\right)\,\Gamma\left(\frac{q+1}{2}\right)}{2\,\Gamma\left(\frac{N+q}{2}\right)}
=
\pi^{\frac{N-1}{2}}\,
\frac{\Gamma\left(\frac{q+1}{2}\right)}{\Gamma\left(\frac{N+q}{2}\right)}.
\end{equation*}
The last part of the statement follows either by the Dominated Convergence Theorem or by the known continuity properties of the Gamma function.
\end{proof}

Concerning the constant in~\eqref{eq:Lambda_sp}, we can prove the following result.

\begin{lemma}
\label{res:lambada}
For $s\in(0,1)$ and $p\in(1,\infty)$ such that $sp<1$, it holds that
\begin{equation}
\label{eq:Lambda_compare}
\Lambda_{s,p}<\Lambda_{sp,1}=\frac{4}{sp}.
\end{equation}
As a consequence, for each $s\in(0,1)$ it holds that
\begin{equation}
\label{eq:Lambda_lim1}
\lim_{p\to1^+}
\Lambda_{s,p}
=
\Lambda_{s,1}.
\end{equation}
\end{lemma}

\begin{proof}
Since $(1-t)^p<1-t^p$ for all $t\in(0,1)$,
we can estimate
\begin{equation*}
\int_0^1\frac{\big|1-t^{\frac{sp-1}{p}}\big|^p}{(1-t)^{1+sp}}\di t
=
\int_0^1\frac{t^{sp-1}\,\big(1-t^{\frac{1-sp}{p}}\big)^p}{(1-t)^{1+sp}}\di t
<
\int_0^1\frac{t^{sp-1}\,\big(1-t^{1-sp}\big)}{(1-t)^{1+sp}}\di t.
\end{equation*}
Letting $q=sp\in(0,1)$, we can write
\begin{equation*}
\begin{split}
\int_0^1\frac{t^{q-1}\,\big(1-t^{1-q}\big)}{(1-t)^{1+q}}\di t
&=
\lim_{\varepsilon\to0^+}
\left(
\mathrm B(1-\varepsilon;q,-q)
-
\int_0^{1-\varepsilon}\frac{1}{(1-t)^{1+q}}\di t
\right)
\\
&=
\frac1q
+
\lim_{\varepsilon\to0^+}
\left(
\mathrm B(1-\varepsilon;q,-q)
-
\frac{\varepsilon^{-q}}{q}
\right),
\end{split}
\end{equation*}
where, for $\varepsilon\in(0,1)$, $\mathrm B(1-\varepsilon;q,-q)$ is as in~\eqref{eq:ibeta}.
In view of~\eqref{eq:ibeta_sym}, we have  
\begin{equation*}
\mathrm B(1-\varepsilon;q,-q)
=
\frac{(1-\varepsilon)^q}{q\,\varepsilon^q}
\end{equation*}
and thus we can easily compute
\begin{equation*}
\lim_{\varepsilon\to0^+}
\left(
\mathrm B(1-\varepsilon;q,-q)
-
\frac{\varepsilon^{-q}}{q}
\right)
=
\lim_{\varepsilon\to0^+}
\left(
\frac{(1-\varepsilon)^q}{q\,\varepsilon^q}
-
\frac{\varepsilon^{-q}}{q}
\right)
=0.
\end{equation*}
We hence get that 
\begin{equation*}
\int_0^1\frac{t^{q-1}\,\big(1-t^{1-q}\big)}{(1-t)^{1+q}}\di t
=
\frac1q
\end{equation*}
and the validity of~\eqref{eq:Lambda_compare} immediately follows by recalling the definition in~\eqref{eq:Lambda_sp}.
By Fatou's Lemma and by~\eqref{eq:Lambda_compare}, we infer that
\begin{equation*}
\Lambda_{s,1}
\le 
\liminf_{p\to1^+}
\Lambda_{s,p}
\le
\limsup_{p\to1^+}
\Lambda_{s,p}
\le 
\limsup_{p\to1^+}
\Lambda_{sp,1}
=
\lim_{p\to1^+}
\frac{4}{sp}
=
\frac4s=\Lambda_{s,1},
\end{equation*}
proving~\eqref{eq:Lambda_lim1} and concluding the proof.
\end{proof}

\subsection{Rearrangements on the real line}

Given $s\in(0,1)$ and a (non-empty) open set $\Omega\subsetneq\R^N$, we let $\delta_{s,\Omega}\colon\R^N\to[0,\infty]$ be given by 
\begin{equation}
\label{eq:delta_func}
\delta_{s,\Omega}(x)
=
\frac{\mathbf 1_\Omega(x)}{d_\Omega^s(x)}
\quad
\text{for all}\ x\in\R^N,
\end{equation}
where $d_\Omega$ is the distance function from $\partial\Omega$ defined in~\eqref{eq:dist_func}.

In the following result, we study the symmetric decreasing rearrangement of the function in~\eqref{eq:delta_func} in the case $N=1$ and $\Omega$ is a finite disjoint union of bounded open intervals.

\begin{proposition}
\label{res:nomega}
Let $N=1$, $s\in(0,1)$ and $\Omega\subsetneq\R$ be the disjoint union of $n\in\mathbb N$ (non-empty) bounded open intervals $\{I_k:k=1,\dots,n\}$ and define
\begin{equation*}
r=\frac12\max\limits_{k=1,\dots,n}|I_k|.
\end{equation*}
Then, the function $\delta_{s,\Omega}\colon\R\to[0,\infty]$, defined in~\eqref{eq:delta_func}, satisfies
\begin{equation}
\label{eq:mariemonti}
|\{\delta_{s,I}>t\}|
\le
2nr\,\mathbf 1_{(0,r^{-s}]}(t)
+
2nt^{-\frac1s}\,\mathbf 1_{(r^{-s},\infty)}(t),
\quad
\text{for all}\ t>0,
\end{equation} 
and thus its symmetric decreasing rearrangement $\delta^\star_{s,I}\colon\R\to[0,\infty]$ satisfies
\begin{equation}
\label{eq:nomega}
\delta_{s,\Omega}^\star(x)
\le 
\frac{\mathbf 1_{(-rn,rn)}(x)}{|x|^s},
\quad
\text{for all}\ x\in\R.
\end{equation}
Moreover, if $|I_k|=2r$ for all $k=1,\dots,n$, then~\eqref{eq:mariemonti} and~\eqref{eq:nomega} are both equalities. 
\end{proposition}

In order to prove \cref{res:nomega}, we start with the case of a single interval.

\begin{lemma}
\label{res:rear}
Let $s\in(0,1)$, $N=1$ and $I\subsetneq\R$ be a non-empty bounded open interval.
Then, the function 
$\delta_{s,I}\colon\R\to[0,\infty]$, defined in~\eqref{eq:delta_func},
satisfies
\begin{equation}
\label{eq:rear_meas}
|\{\delta_{s,I}>t\}|
=
|I|\,\mathbf 1_{(0,2^s|I|^{-s}]}(t)
+
2\,t^{-\frac1s}
\,
\mathbf 1_{(2^s|I|^{-s},\infty)}(t),
\quad
\text{for all}\
t>0,
\end{equation}
and thus its symmetric decreasing rearrangement $\delta_{s,I}^\star\colon\R\to[0,\infty]$ is given by
\begin{equation}
\label{eq:rear}
\delta_{s,I}^\star(x)
=
\frac{\mathbf 1_{I^\star}(x)}{|x|^s},
\quad
\text{for all}\
x\in\R.
\end{equation}
\end{lemma}

\begin{proof}
Let $r\in(0,\infty)$ and $c\in\R$ be such that $I=(c-r,c+r)$, so that $d_I(x)=r-|x-c|$ for all $x\in I$.
Since $\delta_{s,I}(x)>t$ if and only if $|x-c|>r-t^{-\frac1s}$ for all $x\in I$, we have
\begin{equation*}
\{\delta_{s,I}>t\}
=
\begin{cases}
I
&
0<t\le r^{-s},
\\[1ex]
\left(c-r,c-r+t^{-\frac1s}\right)
\cup
\left(c+r-t^{-\frac1s},c+r\right)
&
t>r^{-s},
\end{cases}
\end{equation*}
from which we readily deduce~\eqref{eq:rear_meas}.
We hence get that
\begin{equation*}
\{\delta_{s,I}>t\}^\star
=
\begin{cases}
I^\star
&
0<t\le r^{-s},
\\[1ex]
\left(-t^{-\frac1s},t^{-\frac1s}\right)
&
t>r^{-s},
\end{cases}
\end{equation*}
and thus, recalling the definition in~\eqref{eq:sdr}, we can compute
\begin{equation}
\label{eq:aliante}
\begin{split}
\delta_{s,I}^\star(x)
&=
\int_0^\infty\mathbf 1_{\{\delta_{s,I}>t\}^\star}(x)\di t
=
\int_0^{r^{-s}}
\mathbf 1_{I^\star}(x)\di t
+
\int_{r^{-s}}^\infty
\mathbf 1_{\left(-t^{-1/s},t^{-1/s}\right)}(x)\di t
\\
&=
r^{-s}\,
\mathbf 1_{I^\star}(x)
+
\mathbf 1_{I^\star}(x)
\int_{r^{-s}}^\infty
\mathbf 1_{\left(0,|x|^{-s}\right)}(t)
\di t=
|x|^{-s}\,\mathbf 1_{I^\star}(x)
\end{split}
\end{equation}
for all $x\in\R$, proving~\eqref{eq:rear} and concluding the proof.
\end{proof}

\begin{remark}
\label{rem:reareq}
With the same notation of \cref{res:rear}, if $I,J\subseteq\R$ are two non-empty bounded open intervals such that $|I|=|J|$, then $\delta_{s,I}^\star=\delta_{s,J}^\star$ for all $s\in(0,1)$.
\end{remark}

\begin{remark}
\label{rem:rear2}
\cref{res:rear} also holds for $N\ge2$, but the formula corresponding to~\eqref{eq:rear} has a completely different appearance. 
More precisely, taking $\Omega=\mathbb B^N$ for simplicity in~\eqref{eq:delta_func},
analogous computations yields that
\begin{equation*}
\delta_{s,\mathbb B^N}^\star(x)
=
\mathbf 1_{\mathbb B^N}(x)
\,
\left(1-\left(1-|x|^N\right)^\frac1N\right)^{-s}
\end{equation*}
for all $x\in\R^N$.
We omit the detailed derivation.
\end{remark}

\begin{proof}[Proof of \cref{res:nomega}]
As the $I_k$'s are disjoint, we can write 
\begin{equation}
\label{eq:deco}
\delta_{s,\Omega}(x)
=
\sum_{k=1}^n
\delta_{s,I_k}(x)
\end{equation}
for all $x\in\R$.
The decomposition in~\eqref{eq:deco} implies that
\begin{equation*}
\{\delta_{s,\Omega}>t\}
=
\bigcup_{k=1}^n
\{\delta_{s,I_k}>t\}
\end{equation*}
for all $t>0$ with disjoint union.
By \cref{res:rear}, we hence get that
\begin{equation}
\label{eq:mozzarella}
|\{\delta_{s,\Omega}>t\}|
=
\sum_{k=1}^n|\{\delta_{s,I_k}>t\}|
=
\sum_{k=1}^n
|I_k|\,\mathbf 1_{\left(0,2^s|I_k|^{-s}\right]}(t)
+
2t^{-\frac1s}
\,
\mathbf 1_{\left(2^s|I_k|^{-s},\infty\right)}(t)
\end{equation}
for all $t>0$.
Now let us set $r_k=|I_k|/2$ for all $k=1,\dots,n$ for notational convenience.
Without loss of generality, we can assume that $0< r_1\le r_2\le\dots\le r_n<\infty$.
Furthermore, let us set $r_0=0$ and $r_{n+1}=\infty$ and adopt the notation $r^{-s}_0=\infty$ and $r_{n+1}^{-s}=0$.
With this conventions in force, we can write
\begin{equation*}
\mathbf 1_{\left(0,r_k^{-s}\right]}
=
\sum_{j=0}^{n-k}
\mathbf 1_{\left(r_{n+1-j}^{-s},r_{n-j}^{-s}\right]}
\quad
\text{and}
\quad
\mathbf 1_{\left(r_k^{-s},\infty\right)}
=
\sum_{j=0}^{n-1}
\mathbf 1_{\left(r_{k-j}^{-s},r_{k-j-1}^{-s}\right]}
\end{equation*}
for all $k=1,\dots,n$.
Therefore, on the one hand, we get
\begin{equation}
\label{eq:kiwi1}
\sum_{k=1}^n
r_k\,\mathbf 1_{\left(0,r_k^{-s}\right]}
=
\sum_{j=0}^{n-k}
\mathbf 1_{\left(r_{n+1-j}^{-s},r_{n-j}^{-s}\right]}
\left(
\sum_{k=1}^{n-j}r_k
\right)
\end{equation}
for all $k=1,\dots,n$, while, on the other hand, we can compute
\begin{equation}
\label{eq:kiwi2}
\begin{split}
\sum_{k=1}^n
\mathbf 1_{\left(r_k^{-s},\infty\right)}
&=
\sum_{k=1}^n
\sum_{j=0}^{n-1}
\mathbf 1_{\left(r_{k-j}^{-s},r_{k-j-1}^{-s}\right]}
=
\sum_{j=0}^{n-1}
\sum_{k=j+1}^n
\mathbf 1_{\left(r_{k-j}^{-s},r_{k-j-1}^{-s}\right]}
=
\sum_{j=0}^{n-1}
\sum_{h=1}^{n-j}
\mathbf 1_{\left(r_{h}^{-s},r_{h-1}^{-s}\right]}
\\
&=
\sum_{h=1}^{n}
\sum_{j=0}^{n-h}
\mathbf 1_{\left(r_{h}^{-s},r_{h-1}^{-s}\right]}
=
\sum_{h=1}^{n}
(n-h+1)\,
\mathbf 1_{\left(r_{h}^{-s},r_{h-1}^{-s}\right]}
=
\sum_{j=0}^{n}
j\,
\mathbf 1_{\left(r_{n+1-j}^{-s},r_{n-j}^{-s}\right]}
\end{split}
\end{equation}
for all $k=1,\dots,n$.
By combining~\eqref{eq:kiwi1} and~\eqref{eq:kiwi2} with~\eqref{eq:mozzarella}, we conclude that 
\begin{equation}
\label{eq:devilmeas}
|\{\delta_{s,\Omega}>t\}|
=
2\sum_{j=0}^n
\mathbf 1_{\left(r^{-s}_{n+1-j},r^{-s}_{n-j}\right]}(t)
\left(j\,t^{-\frac1s}
+
\sum_{k=1}^{n-j} r_k
\right)
\end{equation}
for all $t>0$.
Note that, for $r_{n+1-j}^{-s}<t\le r_{n-j}^{-s}$, we have $r_{n-j}\le t^{-\frac1s}<r_{n+1-j}$, and so 
\begin{equation*}
j\,t^{-\frac1s}
+
\sum_{k=1}^{n-j} r_k
\le 
j\,t^{-\frac1s}
+
(n-j)\, r_{n-j}
\le 
j\,t^{-\frac1s}
+
(n-j)\,t^{-\frac1s}
=
n\,t^{-\frac1s}.
\end{equation*}
Thus, we can estimate the measure in~\eqref{eq:devilmeas} (with equality if the $r_k$'s are all equal) as
\begin{equation*}
\begin{split}
|\{\delta_{s,\Omega}>t\}|
&\le 
\mathbf 1_{\left(0,r^{-s}_{n}\right]}(t)
\left(
2\sum_{k=1}^{n} r_k
\right)
+
2\sum_{j=1}^n
\mathbf 1_{\left(r^{-s}_{n+1-j},r^{-s}_{n-j}\right]}(t)
\,n\,t^{-\frac1s}
\\
&\le
2nr_n
\,
\mathbf 1_{\left(0,r^{-s}_{n}\right]}(t)
+
2nt^{-\frac1s}
\,
\mathbf 1_{\left(r^{-s}_{n},\infty\right)}(t)
\end{split}
\end{equation*}    
for all $t>0$, proving~\eqref{eq:mariemonti}.
We thus obtain that (with equality if the $r_k$'s are all equal) 
\begin{equation*}
\{\delta_{s,\Omega}>t\}^\star
\subseteq
\begin{cases}
(-nr_n,nr_n) 
&
0<t\le r^{-s}_n,
\\[1ex]
\left(-nt^{-\frac1s},nt^{-\frac1s}\right)
&
t>r^{-s}_n,
\end{cases}
\end{equation*}
and consequently, arguing similarly as in~\eqref{eq:aliante}, we get that
\begin{equation*}
\begin{split}
\delta_{s,\Omega}^\star(x)
&=
\int_0^\infty\mathbf 1_{\{\delta_{s,\Omega}>t\}^\star}(x)\di t
\le
\int_0^{r_n^{-s}}\mathbf 1_{(-nr_n,nr_n)}(x)\di t
+
\int_{r_n^{-s}}^\infty\mathbf 1_{(-nt^{-1/s},nt^{-1/s})}(x)\di t
\\
&=
r_n^{-s}\,\mathbf 1_{(-nr_n,nr_n)}(x)
+
\mathbf 1_{(-nr_n,nr_n)}(x)
\,
\int_{r_n^{-s}}^\infty\mathbf 1_{(0,n^s|x|^{-s})}(t)\di t
=
n^s\,\mathbf 1_{(-nr_n,nr_n)}(x)\,|x|^{-s}
\end{split}
\end{equation*}
for all $x\in\R$.
Recalling that $r_n\le r$ by assumption, we obtain~\eqref{eq:nomega}.
\end{proof}

We conclude with the following result, which is a direct consequence of \cref{res:nomega}.

\begin{corollary}
\label{res:nomega_vol}
Let $N=1$, $s\in(0,1)$ and $\Omega\subsetneq\R$ be the disjoint union of $n\in\mathbb N$ (non-empty) bounded open intervals.
If $u\in C^\infty_0(\Omega)$ and $p\in[1,\infty)$ is such that $sp<1$, then 
\begin{equation}
\label{eq:aceto}
\int_{\Omega}
\frac{|u|^p}{d_\Omega^{sp}}\di x
\le
n^{sp} 
\int_{\R} 
\frac{(u^\star)^p}{|x|^{sp}}
\di x.
\end{equation}
\end{corollary}

\begin{proof}
Given $u\in C^\infty_0(\Omega)$,  recalling~\eqref{eq:delta_func} and owing to~\eqref{eq:hl} and~\eqref{eq:circostar}, we can estimate
\begin{equation*}
\int_{\Omega}\frac{|u|^p}{d_\Omega^{sp}}\di x
=
\int_\R
|u|^p
\,
\delta_{sp,\Omega}\di x
\le 
\int_\R
(|u|^p)^\star
\,
\delta_{sp,\Omega}^\star\di x
=
\int_\R
(u^\star)^p
\,
\delta_{sp,\Omega}^\star\di x.
\end{equation*}
Therefore, since $sp<1$, we can exploit~\eqref{eq:nomega} (with $sp$ in place of $s$) to infer that 
\begin{equation*}
\int_{\Omega}\frac{|u|^p}{d_\Omega^{sp}}\di x
\le 
\int_\R
(u^\star)^p
\,
\delta_{sp,\Omega}^\star\di x
\le
n^{sp}
\int_{\R}
\frac{(u^\star)^p}{|x|^{sp}}
\di x,
\end{equation*}
proving~\eqref{eq:aceto} and concluding the proof.
\end{proof}

\section{The geometrical approach}

In this section, we detail the proof of \cref{res:cheeger} and derive some consequences, both in the convex and non-convex case.

\subsection{Proof of \texorpdfstring{\cref{res:cheeger}}{Theorem 1.4}}

We begin with our auxiliary result concerning the geometrical interpretation of the sharp constant~\eqref{eq:h_s1}.

\begin{proof}[Proof of \cref{res:cheeger}]
Let us observe that $\mathfrak g_s(\Omega)\in[0,\infty)$ is well posed, since
\begin{equation*}
\mathfrak g_s(\Omega)
\le
\frac{P_s(B)}{V_{s,\Omega}(B)}<\infty
\end{equation*}
whenever $B\Subset\Omega$ is a (non-empty) open ball.
We start by showing that
\begin{equation}
\label{eq:hgeg} 
\mathfrak{h}_{s,1}(\Omega) 
\ge 
\mathfrak{g}_s(\Omega).
\end{equation}
We fix $u\in C^{\infty}_0(\Omega)$ and, without loss of generality, assume that $u\ge0$.
By Cavalieri's principle, we have that
\begin{equation}
\label{eq:cavalieri}
\int_\Omega\frac{u}{d^s_\Omega}\di x
=
\int_0^{\|u\|_{L^\infty(\Omega)}}
\int_{\{u>t\}}\frac{1}{d^s_\Omega}\di x\di t
=
\int_0^{\|u\|_{L^\infty(\Omega)}}
V_{s,\Omega}(\{u>t\})\di t.
\end{equation}
Now, for each $t\in(0,\|u\|_{L^\infty(\Omega)})$, the (non-empty, open) set $E_t=\{u>t\}$ is such that $E_t\Subset\Omega$ and $|E_t|>0$ and thus, by the definition in~\eqref{eq:cheeger}, we can write
\begin{equation*}
P_s(E_t)\ge \mathfrak g_s(\Omega)\, V_{s,\Omega}(E_t).
\end{equation*}
Hence, thanks to the fractional coarea formula in~\eqref{eq:fracoarea}, we infer that 
\begin{equation*}
[u]_{W^{s,1}(\R^N)}
=
\int_0^{\|u\|_{L^\infty(\Omega)}}
P_s(E_t)\di t
\ge 
\int_0^{\|u\|_{L^\infty(\Omega)}}
\mathfrak g_s(\Omega)\, V_{s,\Omega}(E_t)
\di t
=
\mathfrak g_s(\Omega)
\int_\Omega\frac{u}{d^s_\Omega}\di x,
\end{equation*}
readily implying the validity of~\eqref{eq:hgeg}. 
We now prove the converse inequality; that is, 
\begin{equation}
\label{eq:ggeh}
\mathfrak{h}_{s,1}(\Omega) 
\le 
\mathfrak{g}_s(\Omega). 
\end{equation}
We let $E\Subset\Omega$ be such that $|E|>0$, and we define
\begin{equation*}
u_{\varepsilon}=\varrho_\varepsilon*\mathbf 1_{E}\in C^\infty(\R^N),
\end{equation*}
where $(\varrho_\varepsilon)_{\varepsilon>0}\subseteq C^\infty_0(\R^N)$ is a family of standard non-negative mollifiers.
Since $E\Subset\Omega$, we  have that $\operatorname{supp}u_{\varepsilon}\Subset\Omega$ for all $\varepsilon>0$ sufficiently small.
Therefore, owing to~\cite{L23}*{Th.~6.62} and the Dominated Convergence Theorem, we have that
\begin{equation*}
\mathfrak h_{s,1}(\Omega)
\le
\lim_{\varepsilon\to0^+}
\frac{[u_{\varepsilon}]_{W^{s,1}(\R^N)}}{\displaystyle\int_{\Omega}\frac{u_{\varepsilon}}{d^{s}_\Omega}\di x}
=
\frac{P_s(E)}{V_{s,\Omega}(E)},
\end{equation*}
from which we readily deduce~\eqref{eq:ggeh} and reach the conclusion of the proof.
\end{proof}

Actually, in \cref{res:cheeger}, we can restrict the infimization in~\eqref{eq:cheeger} only to smooth open sets compactly contained in $\Omega$.
Precisely, we have the following result. 

\begin{corollary}
\label{res:cheeger_smooth}
Let $N\ge1$, $s\in(0,1)$ and let $\Omega\subsetneq\R^N$ be a (non-empty) open set.
Then, 
\begin{equation}
\label{eq:cheeger_smooth}
\mathfrak g_s(\Omega)
=
\inf\left\{
\frac{P_s(E)}{V_{s,\Omega}(E)}
:
E\in\mathcal O(\Omega)
\right\}\in[0,\infty),
\end{equation}
where $\mathcal O(\Omega)$ is the family of all open sets with smooth boundary compactly contained in~$\Omega$.
\end{corollary}

\begin{proof}
Letting $\widetilde{\mathfrak g}_s(\Omega)\in[0,\infty]$ be the right-hand side of~\eqref{eq:cheeger_smooth}, we have $\mathfrak g_s(\Omega)\le \widetilde{\mathfrak g}_s(\Omega)$ by definition.
On the other hand, if $u\in C^\infty_0(\Omega)$ is such that $u\ge0$, then $\{u>t\}\in\mathcal O(\Omega)$ for a.e.\ $t\in\left(0,\|u\|_{L^\infty}\right)$ (e.g., see~\cite{M12}*{Th.~13.15}). 
Hence, arguing exactly as in the first part of the proof of \cref{res:cheeger}, we see that $\mathfrak h_{s,1}(\Omega)\ge\widetilde{\mathfrak g}_s(\Omega)$, in particularly yielding that $\widetilde{\mathfrak g}_s(\Omega)<\infty$.
Thus we have $\widetilde{\mathfrak g}_s(\Omega)\le \mathfrak g_s(\Omega)$ by \cref{res:cheeger}, leading to the conclusion.
\end{proof}

A remarkable consequence of \cref{res:cheeger} is the following comparison result.

\begin{corollary}
\label{res:spacco}
Let $N\ge1$, $p\in(1,\infty)$ and $s\in(0,1)$ be such that $sp<1$. 
For any non-empty open set $\Omega\subsetneq\R^N$, it holds that
\begin{equation}
\label{eq:spacco}
\mathfrak h_{s,p}(\Omega)\le\mathfrak h_{sp,1}(\Omega).
\end{equation}
If $\Omega=\mathbb H^N_+$, then the inequality~\eqref{eq:spacco} is strict.
\end{corollary}

\begin{proof}
As in the second part of the proof of \cref{res:cheeger}, given $E\Subset\Omega$ such that $|E|>0$, we define $u_\varepsilon=\mathbf 1_E*\varrho_\varepsilon$ for all $\varepsilon>0$, where $(\varrho_\varepsilon)_{\varepsilon>0}\subseteq C^\infty_0(\R^N)$ is a family of standard non-negative mollifiers.
As before, $u_\varepsilon\in C^\infty_0(\R^N)$ and $0\le u_\varepsilon\le 1$ for all $\varepsilon>0$, with $\operatorname{supp} u_\varepsilon\Subset\Omega$ for all $\varepsilon>0$ sufficiently small.
Therefore, by~\cite{L23}*{Th.~6.62} and the Dominated Convergence Theorem, and since $sp<1$, we get that
\begin{equation*}
\mathfrak h_{s,p}(\Omega)
\le 
\lim_{\varepsilon\to0^+}
\frac{[u_\varepsilon]^p_{W^{s,p}(\R^N)}}{\displaystyle\int_\Omega\frac{u^p_\varepsilon}{d_\Omega^{sp}}\di x}
=
\frac{[\mathbf 1_E]^p_{W^{s,p}(\R^N)}}{\displaystyle\int_\Omega\frac{\mathbf 1_E^p}{d_\Omega^{sp}}\di x}
=
\frac{[\mathbf 1_E]_{W^{sp,1}(\R^N)}}{\displaystyle\int_\Omega\frac{\mathbf 1_E}{d_\Omega^{sp}}\di x}
=
\frac{P_{sp}(E)}{V_{sp,\Omega}(E)}.
\end{equation*}  
Consequently, by \cref{res:cheeger}, we conclude that
\begin{equation*}
\mathfrak h_{s,p}(\Omega)
\le 
\mathfrak g_{sp,1}(\Omega)
=
\mathfrak h_{sp,1}(\Omega),
\end{equation*}
proving~\eqref{eq:spacco}. 
For $\Omega=\mathbb H^N_+$, the inequality in~\eqref{eq:spacco} holds strict by~\eqref{eq:hs1_H} and~\eqref{eq:Lambda_compare}.
\end{proof}

\subsection{The convex case}

If $\Omega\subsetneq\R^N$ is a (non-empty) convex open set, then we can refine the statement of \cref{res:cheeger} as follows.

\begin{corollary}
\label{res:cheeger_relax}
Let $N\ge1$, $s\in(0,1)$ and let $\Omega\subsetneq\R^N$ be a (non-empty) open convex set.
Then, it holds that
\begin{equation}
\label{eq:h_s1=g_s*}
\mathfrak h_{s,1}(\Omega)
=
\mathfrak g_{s}^*(\Omega),
\end{equation}
where
\begin{equation}
\label{eq:g_s*}
\mathfrak g_s^*(\Omega)
=
\inf \left\{ \frac{P_s(E)}{V_{s,\Omega}(E)} : E \subseteq\Omega,\ |E|\in(0,\infty) \right\}\in(0,\infty).
\end{equation}
Moreover, if $u\in\mathcal W^{s,1}_0(\Omega)\setminus\{0\}$ is a non-negative minimizer of~\eqref{eq:h_s1} then, for a.e.\ $t>0$ such that $|\{u>t\}|\in(0,\infty)$, $E_t=\{u>t\}$ is a minimizer of~\eqref{eq:g_s*}.
Vice versa, if $E\subseteq\Omega$ is a minimizer of~\eqref{eq:g_s*} such that $\mathbf 1_E\in\mathcal W^{s,1}_0(\Omega)$, then $u=\mathbf 1_E/V_{s,\Omega}(E)$ is a non-negative minimizer of~\eqref{eq:h_s1}.  
\end{corollary}

In the proof of \cref{res:cheeger_relax}, we are going to exploit the following simple approximation result, whose proof is omitted. 

\begin{lemma}
\label{res:approx}
Let $N\ge1$.
If $\Omega\subseteq\R^N$ is a (non-empty) open convex set, then there exists a sequence of (non-empty) bounded open convex sets $(\Omega_k)_{k\in\mathbb N}$ such that 
\begin{equation*}
\Omega_k\Subset\Omega_{k+1}\Subset\Omega
\quad
\text{for all}\ k\in\mathbb N
\qquad 
\text{and}\qquad
\Omega=\bigcup_{k=1}^\infty\Omega_k.
\end{equation*}
\end{lemma}

\begin{proof}[Proof of \cref{res:cheeger_relax}]
By \cref{res:cheeger} and the trivial inequality $\mathfrak g_s^*(\Omega)\le\mathfrak g_s(\Omega)$, to show~\eqref{eq:h_s1=g_s*} we only have to prove that $\mathfrak g_s(\Omega)\le\mathfrak g_s^*(\Omega)$. 
To this aim, let $E\subseteq\Omega$ be such that $|E|\in(0,\infty)$, $(\Omega_k)_{k\in\mathbb N}$ be the sequence of (non-empty) bounded open convex set given by \cref{res:approx}, and  $E_k=E\cap\Omega_k$ for every $k\in\mathbb N$. 
Clearly, for all $k\in\mathbb N$ sufficiently large, we have that $E_k\Subset\Omega$ and $|E_k|\in(0,\infty)$.
Moreover, by the Monotone Convergence Theorem, we also have that 
\begin{equation*}
V_{s,\Omega}(E)
=
\lim_{k\to\infty}
V_{s,\Omega}(E_k).
\end{equation*}
Finally, by~\eqref{eq:interconvex}, we also have that $P_s(E_k)\le P_s(E)$ for all $k\in\mathbb N$. 
Thus, we get 
\begin{equation*}
\mathfrak g_s(\Omega)
\le 
\limsup_{k\to\infty}
\frac{P_s(E_k)}{V_{s,\Omega}(E_k)}
\le 
\lim_{k\to\infty}
\frac{P_s(E)}{V_{s,\Omega}(E_k)}
=
\frac{P_s(E)}{V_{s,\Omega}(E)},
\end{equation*} 
from which~\eqref{eq:h_s1=g_s*} follows immediately.
By \cref{res:hs1}, this also shows that $\mathfrak g_s^*(\Omega)>0$.

To prove the last part of the statement, it is enough to observe that, if $u\in\mathcal W^{s,1}_0(\Omega)\setminus\{0\}$ is a non-negative minimizer of~\eqref{eq:h_s1} and $E_t=\{u>t\}$ for $t>0$, then, by~\eqref{eq:g_s*}, the fractional coarea formula~\eqref{eq:fracoarea} and Cavalieri's principle~\eqref{eq:cavalieri}, we can write
\begin{equation*}
[u]_{W^{s,1}(\R^N)}
-
\mathfrak h_{s,1}(\Omega)
\int_\Omega\frac{u}{d_\Omega^s}\di x
=
\int_0^\infty
\Big(
P_s(E_t)
-
\mathfrak g_{s}^*(\Omega)
\,
V_{s,\Omega}(E_t)
\Big)\di t,
\end{equation*}
from which we deduce that $E_t$ is a minimizer of~\eqref{eq:g_s*} for a.e.\ $t>0$ such that $|E_t|\in(0,\infty)$.
Vice versa, if $E\subseteq\Omega$ is a minimizer of~\eqref{eq:g_s*} such that $\mathbf 1_E\in\mathcal W^{s,1}_0(\Omega)$, then clearly also $u=\mathbf 1_E/V_{s,\Omega}(E)\in\mathcal W^{s,1}_0(\Omega)\setminus\{0\}$.
Moreover, by definition and~\eqref{eq:h_s1=g_s*}, we have that
\begin{equation*}
[u]_{W^{s,1}(\R^N)}
=
\frac{P_s(E)}{V_{s,\Omega}(E)}
=
\mathfrak g_s^*(\Omega)=\mathfrak h_{s,1}(\Omega)
\quad
\text{and}
\quad
\int_\Omega\frac{u}{d_\Omega^s}
\di x
=
1,
\end{equation*}
proving that~$u$ is a minimizer of~\eqref{eq:h_s1}. 
\end{proof}

As an immediate consequence of \cref{res:cheeger_relax}, we get the following estimate.

\begin{corollary}
\label{res:cheeger_omega}
Let $N\ge1$ and $s\in(0,1)$.
If $\Omega\subsetneq\R^N$ is a (non-empty) bounded open convex set, then
\begin{equation*}
\mathfrak h_{s}(\Omega)
\le 
\frac{P_s(\Omega)}{V_{s,\Omega}(\Omega)}.
\end{equation*}
\end{corollary}

\subsection{The non-convex case}

A statement similar to \cref{res:cheeger_relax} can be achieved also in the case $\Omega$ is a non-convex set, see \cref{res:cheeger_s-neg} below.
To this purpose, we need to introduce the following terminology.

\begin{definition}[$s$-negligible set]
Let $s\in(0,1)$.
A set $E\subseteq\R^N$ is \emph{$s$-negligible} if 
\begin{equation*}
\lim_{r\to0^+}
P_s
\big(
B_r(E)
\big)
=0,
\end{equation*}
where $B_r(E)=\bigcup_{x\in E}B_r(x)$ for every $r>0$.
\end{definition}

For example, by~\eqref{eq:submodularity} and the fact that $P_s(B_r(x))=P_s(\mathbb B^N)\,r^{N-s}$ for all $x\in\R^N$ and $r>0$, any set $E\subset\R^N$ with finitely many points is $s$-negligible for every $s\in(0,1)$.

\begin{corollary}
\label{res:cheeger_s-neg}
Let $N\ge1$, $s\in(0,1)$, $U\subseteq\R^N$ be a (non-empty) open convex set and let $F\subsetneq U$ be an $s$-negligible closed set.
If $\Omega=U\setminus F$, then
\begin{equation}
\label{eq:cheeger_s-neg}
\mathfrak h_{s,1}(\Omega)
\le 
\inf\left\{
\frac{P_s(E)}{V_{s,\Omega}(E)} : E\subseteq U,\ |E|\in(0,\infty)
\right\}.
\end{equation} 
In particular, if $U$ is bounded, then
\begin{equation*}
\mathfrak h_{s,1}(\Omega)
\le 
\frac{P_s(U)}{V_{s,\Omega}(U)}.
\end{equation*}
\end{corollary}

\begin{proof}
Let $(U_k)_{k\in\mathbb N}$ be the sequence of (non-empty) bounded open convex sets given by \cref{res:approx} applied to~$U$.
Given $E\subseteq U$ such that $|E|\in(0,\infty)$, we let $r>0$ and consider 
\begin{equation*}
E_k=(E\cap U_k)
\quad
\text{and}
\quad
E_{k,r}=E_k\setminus
B_r(F).
\end{equation*}
Note that $E_{k,r}\Subset\Omega$ and $|E_{k,r}|\in(0,\infty)$ for every $k\in\mathbb N$ sufficiently large and every $r>0$ sufficiently small. 
Now, on the one hand, by~\eqref{eq:perdiff}, \eqref{eq:interconvex} and~\eqref{eq:submodularity}, we can estimate
\begin{equation*}
P_s(E_{k,r})
\le 
P_s(E_k)
+
P_s\big(
B_r(F)
\big)
\le 
P_s(E)
+
P_s\big(
B_r(F)
\big).
\end{equation*}
On the other hand, since $E_{k,r}\subseteq  E_k\subseteq E$ for $k\in\mathbb N$ and $r>0$, by the Monotone Convergence Theorem, we infer that
\begin{equation*}
\lim_{k\to\infty}
\lim_{r\to0^+}
V_{s,\Omega}(E_{k,r})
=
\lim_{k\to\infty}
V_{s,\Omega}(E_{k})
=
V_{s,\Omega}(E).
\end{equation*}
Therefore, owing to \cref{res:cheeger} and since $F$ is $s$-negligible, we conclude that
\begin{equation*}
\mathfrak h_{s,1}(\Omega)
=
\mathfrak g_s(\Omega)
\le 
\limsup_{k\to\infty}
\limsup_{r\to0^+}
\frac{P_s(E_{k,r})}{V_{s,\Omega}(E_{k,r})}
\le
\lim_{k\to\infty}
\lim_{r\to0^+}
\frac{P_s(E)+P_s\big(B_r(F)\big)}{V_{s,\Omega}(E_{k,r})}
=
\frac{P_s(E)}{V_{s,\Omega}(E)}
,
\end{equation*}
from which~\eqref{eq:cheeger_s-neg} follows immediately. 
\end{proof}

\section{Proofs of the results}

\subsection{Proof of \texorpdfstring{\cref{res:hs1}}{Theorem 1.1}}

In order to characterize the sharp constant in~\eqref{eq:h_sp} for $\Omega=\mathbb H^N_+$ and $p=1$, so to prove \cref{res:hs1}, we need some preliminary results.

We begin with the following stability result for $\mathfrak h_{s,p}(\Omega)$ as $p\to1^+$.

\begin{lemma}
\label{res:stability}
Let $N\ge1$, $s\in(0,1)$ and $\Omega\subsetneq\R^N$ be a (non-empty)  open set.
Then,
\begin{equation}
\label{eq:stab_dist}
\lim_{p \to 1^+} 
\int_{\Omega} \frac{|u|^p}{d_{\Omega}^{sp}}
\di x
= 
\int_{\Omega} \frac{|u|}{d_{\Omega}^{s}}\di x
\end{equation}
and
\begin{equation}
\label{eq:stab_snorm}
\lim_{p\to1^+}\,
[u]_{W^{s,p}(\R^N)}^p
= 
[u]_{W^{s,1}(\R^N)}
\end{equation}
for every $u\in C^\infty_0(\Omega)$.
As a consequence, it holds that
\begin{equation}
\label{eq:stab_h}
\limsup_{p\to1^+}\,
\mathfrak h_{s,p}(\Omega)
\le 
\mathfrak h_{s,1}(\Omega).
\end{equation}
\end{lemma}

\begin{proof}
For $p\in(1,\infty)$, by H\"older's inequality, we can estimate
\begin{equation}
\label{eq:ping}
\int_{\Omega} 
\frac{|u|}{d_{\Omega}^{s}}
\di x
=
\int_{\operatorname{supp} u} 
\frac{|u|}{d_{\Omega}^{s}}
\di x 
\le 
\left(
\int_{\Omega} \frac{|u|^p}{d_{\Omega}^{sp}}\di x 
\right)^{\frac{1}{p}} 
|\operatorname{supp} u|^{\frac{p-1}{p}},
\end{equation}
and, similarly,
\begin{equation}
\label{eq:pong}
\int_{\Omega} 
\frac{|u|^p}{d_{\Omega}^{sp}}\di x 
\le 
\left\| \frac{u}{d_{\Omega}^{s}}\right\|_{L^{\infty}(\Omega)}^{p-1} 
\int_{\Omega} 
\frac{|u|}{d_{\Omega}^{s}}
\di x.
\end{equation}
Hence the validity of~\eqref{eq:stab_dist} immediately follows by combining~\eqref{eq:ping} and~\eqref{eq:pong}.
Concerning~\eqref{eq:stab_snorm}, owing to Fatou's Lemma, we just need to observe that 
\begin{equation*}
[u]_{W^{s,p}(\mathbb{R}^N)}^p 
= 
\int_{\mathbb{R}^N}
\int_{\mathbb{R}^N} 
\frac{|u(x)-u(y)|^{p-1}}{|x-y|^{s(p-1)}} \, \frac{|u(x)-u(y)|}{|x-y|^{N+s}}
\di x\di y 
\le 
[u]_{C^{0,s}(\mathbb{R}^N)}^{p-1} 
\, 
[u]_{W^{s,1}(\mathbb{R}^N)}.
\end{equation*}
Lastly, the validity of~\eqref{eq:stab_h} readily follows by combining~\eqref{eq:h_sp} with~\eqref{eq:stab_dist} and~\eqref{eq:stab_snorm}.
\end{proof}

In the case $\Omega=\mathbb H^N_+$, \cref{res:stability} can be complemented with the following result, yielding useful estimates on the energies of product test functions for $p=1$.

\begin{lemma}
\label{res:prod}
If $\varphi\in C^\infty_0(\R^{N-1})$ and $\psi\in C^\infty_0(\mathbb H^1_+)$, then $u\in C^\infty_0(\mathbb H^N_+)$, defined as $u(x)=\varphi(x')\,\psi(x_N)$ for all $x=(x',x_N)\in\mathbb H^N_+$, satisfies
\begin{equation}
\label{eq:prod_dist}
\int_{\mathbb H^N_+}\frac{|u(x)|}{x_N^s}\di x
=
\|\varphi\|_{L^1(\R^{N-1})}\int_{\mathbb H^1_+}\frac{|\psi(t)|}{t^s}\di t
\end{equation}
and 
\begin{equation}
\label{eq:prod_snorm}
[u]_{W^{s,1}(\R^N)}
\le 
C_{N,s}\,\|\varphi\|_{L^1(\R^{N-1})}\,[\psi]_{W^{s,1}(\R)}
+
2\|\psi\|_{L^1(\mathbb H^1_+)}\,[\varphi]_{W^{s,1}(\R^{N-1})}\int_0^\infty\frac{\di t}{(1+t^2)^{\frac{N+s}2}}.
\end{equation}
\end{lemma}

\begin{proof}
To prove~\eqref{eq:prod_dist}, we simply observe that 
\begin{equation*}
\int_{\mathbb H^N_+}\frac{|u(x)|}{x_N^s}\di x
=
\int_{\R^{N-1}}|\varphi(x')|\di x'\int_0^\infty\frac{|\psi(x_N)|}{x_N^s}\di x_N
=
\|\varphi\|_{L^1(\R^{N-1})}\int_{\mathbb H^1_+}\frac{|\psi(t)|}{t^s}\di t.
\end{equation*} 
To see~\eqref{eq:prod_snorm}, we first estimate
\begin{equation}
\label{eq:ric}
\begin{split}
[u]_{W^{s,1}(\R^N)}
&\le 
\int_{\R^N}
\int_{\R^N}
|\varphi(x')|
\,
\frac{|\psi(x_N)-\psi(y_N)|}{|x-y|^{N+s}}
\di x\di y
\\
&
\quad+
\int_{\R^N}
\int_{\R^N}
|\psi(y_N)|
\,
\frac{|\varphi(x')-\varphi(y')|}{|x-y|^{N+s}}
\di x\di y.
\end{split}
\end{equation}
The first term in the right-hand side of~\eqref{eq:ric} rewrites as
\begin{equation}
\label{eq:clag}
\begin{split}
\int_{\R^N}
&
\int_{\R^N}
|\varphi(x')|
\,
\frac{|\psi(x_N)-\psi(y_N)|}{|x-y|^{N+s}}
\di x\di y
\\
&=
\int_{\R}
\int_{\R}
\frac{|\psi(x_N)-\psi(y_N)|}{|x_N-y_N|^{N+s}}
\int_{\R^{N-1}}
|\varphi(x')|
\int_{\R^{N-1}}
\frac{\di y'}{\left(1+\tfrac{|x'-y'|^2}{|x_N-y_N|^2}\right)^{\frac{N+s}{2}}}
\di x'
\di x_N\di y_N
\\
&=
\int_{\R}
\int_{\R}
\frac{|\psi(x_N)-\psi(y_N)|}{|x_N-y_N|^{1+s}}
\int_{\R^{N-1}}
|\varphi(x')|
\int_{\R^{N-1}}
\frac{\di z'}{\left(1+|z'|\right)^{\frac{N+s}2}}
\di x'
\di x_N\di y_N
\\
&=
C_{N,s}\,\|\varphi\|_{L^1(\R^{N-1})}\,[\psi]_{W^{s,1}(\R)},
\end{split}
\end{equation}
while the second term in~\eqref{eq:ric} similarly corresponds to 
\begin{equation}
\label{eq:gio}
\begin{split}
\int_{\R^N}
&
\int_{\R^N}
|\psi(y_N)|
\,
\frac{|\varphi(x')-\varphi(y')|}{|x-y|^{N+s}}
\di x\di y
\\
&=
\int_{\R^{N-1}}
\int_{\R^{N-1}}
\frac{|\varphi(x')-\varphi(y')|}{|x'-y'|^{N+s}}
\int_{\R}
|\psi(y_N)|
\int_{\R}
\frac{\di x_N}{\left(1+\tfrac{|x_N-y_N|^2}{|x'-y'|^2}\right)^{\frac{N+s}2}}
\di y_N
\di x'\di y'
\\
&=
\int_{\R^{N-1}}
\int_{\R^{N-1}}
\frac{|\varphi(x')-\varphi(y')|}{|x'-y'|^{N-1+s}}
\int_{\R}
|\psi(y_N)|
\int_{\R}
\frac{\di z_N}{(1+|z_N|^2)^{\frac{N+s}2}}
\di y_N
\di x'\di y'
\\
&=
2\,\|\psi\|_{L^1(\mathbb H^1_+)}\,[\varphi]_{W^{s,1}(\R^{N-1})}\,\int_0^\infty\frac{\di t}{(1+t^2)^{\frac{N+s}2}}.
\end{split}
\end{equation}
The conclusion hence readily follows by combining~\eqref{eq:ric}, \eqref{eq:clag} and~\eqref{eq:gio}.
\end{proof}

As a consequence of the above results, we get the following partial step towards~\eqref{eq:hs1_H}.

\begin{corollary}
\label{res:comput_compar}
Given $N\ge1$ and $s\in(0,1)$, it holds that
\begin{equation}
\label{eq:comput_compar}
C_{N,s}\,\Lambda_{s,1}
\le 
\mathfrak h_{s,1}(\mathbb H^N_+)
\le 
C_{N,s}\,\mathfrak h_{s,1}(\mathbb H^1_+).
\end{equation}
\end{corollary}

\begin{proof}
Owing to~\eqref{eq:stab_h}, \eqref{eq:sharpfrac_h}, \cref{res:bambi} and \cref{res:lambada}, we plainly get that
\begin{equation*}
\mathfrak h_{s,1}(\mathbb H^N_+)
\ge 
\limsup_{p\to1^+}
\mathfrak h_{s,p}(\mathbb H^N_+)
=
\lim_{p\to1^+}
C_{N,sp}\,\Lambda_{s,p}
=
C_{N,s}\,\Lambda_{s,1},
\end{equation*}
proving the first inequality in~\eqref{eq:comput_compar}.

To prove the second inequality in~\eqref{eq:comput_compar}, instead, we argue as follows.
Let $\varphi\in C^\infty_0(\R^{N-1})$ and $\psi\in C^\infty_0(\mathbb H^1_+)$ be fixed.
For each $k\in\mathbb N$, we define $\varphi_k(x')=\varphi\left(\frac{x'}k\right)$ for all $x'\in\R^{N-1}$ and $u_k(x)=\varphi_k(x')\,\psi(x_N)$ for all $x=(x',x_N)\in\mathbb \R^{N-1}\times(0,\infty)=\mathbb H^N_+$.
Owing to \cref{res:prod} and observing that
\begin{equation*}
[\varphi_k]_{W^{s,1}(\R^{N-1})}
=
k^{N-1-s}\,[\varphi]_{W^{s,1}(\R^{N-1})}
\quad
\text{and}
\quad
\|\varphi_k\|_{L^1(\R^{N-1})}
=
k^{N-1}
\,
\|\varphi\|_{L^1(\R^{N-1})},
\end{equation*}
we plainly get that 
\begin{equation*}
\begin{split}
\mathfrak h_{s,1}(\mathbb H^N_+)
&\le 
\frac{[u_k]_{W^{s,1}(\R^N)}}{\displaystyle\int_{\mathbb H^N_+}\frac{|u_k(x)|}{x_N^s}\di x}
\\
&\le 
C_{N,s}\,\frac{[\psi]_{W^{s,1}(\R)}}{\displaystyle\int_{\mathbb H^1_+}\frac{|\psi(t)|}{t^s}\di t}
+
\frac{[\varphi_k]_{W^{s,1}(\R^{N-1})}}{\|\varphi_k\|_{L^1(\R^{N-1})}}
\,
\frac{\|\psi\|_{L^1(\mathbb H^1_+)}\displaystyle\int_0^\infty\frac{\di t}{(1+t^2)^{\frac{N+s}2}}}{\displaystyle\int_{\mathbb H^1_+}\frac{|\psi(t)|}{t^s}\di t}
\\
&=
C_{N,s}\,\frac{[\psi]_{W^{s,1}(\R)}}{\displaystyle\int_{\mathbb H^1_+}\frac{|\psi(t)|}{t^s}\di t}
+
\frac{1}{k^s}\,\frac{[\varphi]_{W^{s,1}(\R^{N-1})}}{\|\varphi\|_{L^1(\R^{N-1})}}
\,
\frac{\|\psi\|_{L^1(\mathbb H^1_+)}\displaystyle\int_0^\infty\frac{\di t}{(1+t^2)^{\frac{N+s}2}}}{\displaystyle\int_{\mathbb H^1_+}\frac{|\psi(t)|}{t^s}\di t}
\end{split}
\end{equation*} 
for all $k\in\mathbb N$.
The second inequality in~\eqref{eq:comput_compar} hence follows first by passing to the limit as $k\to\infty$ and then by taking the infimum with respect to $\psi\in C^\infty_0(\mathbb H^1_+)$.
\end{proof}

At this point, the proof of \cref{res:hs1} essentially reduces to a straight computation based on~\eqref{eq:persseg} and the fact that, for every $s\in(0,1)$,
\begin{equation}
\label{eq:volsseg}
V_{s,(a,b)}((a,b))
=
\frac{2^s}{1-s}\,(b-a)^{1-s}
\end{equation}
whenever $a,b\in\R$ are such that $a<b$.

\begin{proof}[Proof of \cref{res:hs1}]
For $N=1$, owing to~\eqref{eq:persseg}, \eqref{eq:volsseg} and \cref{res:cheeger_relax}, we infer that
\begin{equation*}
\mathfrak h_{s,1}(\mathbb H^1_+)
=
\mathfrak g_s^*(\mathbb H^1_+)
\le 
\frac{P_s((0,1))}{V_{s,\mathbb H^1_+}((0,1))}
=
\frac{4}{s}
\end{equation*}
for all $s\in(0,1)$.  
The validity of~\eqref{eq:hs1_H} hence follows owing to~\eqref{eq:Lambda_compare} and~\eqref{eq:comput_compar}. 
Consequently, we get~\eqref{eq:hs1} by simply passing to the limit as $p\to1^+$ in~\eqref{eq:frachardy_estim}.
Finally, we prove that $\mathfrak h_{s,1}(\mathbb H^1_+)=\frac4s$ is attained in $\mathcal W^{s,1}_0(\mathbb H^1_+)$.
Indeed, thanks to \cref{res:cheeger_relax}, it is enough to check that $\mathbf 1_{(0,1)}\in\mathcal W^{s,1}_0(\mathbb H^1_+)$.
This, in turn, follows by observing that 
\begin{equation*}
\lim_{t\to0^+}
[\mathbf 1_{(t,1)}-\mathbf 1_{(0,1)}]_{W^{s,1}(\R)}
=
\lim_{t\to0^+}
P_s((0,t))=0
\end{equation*} 
by~\eqref{eq:persseg}, and that $1_{(t,1)}\in\mathcal W^{s,1}_0(\mathbb H^1_+)$ for every $t\in(0,1)$, since $u_{\varepsilon,t}=1_{(t,1)}*\varrho_\varepsilon\in C^\infty_0(\mathbb H^1_+)$ are converging to $1_{(t,1)}$ in $W^{s,1}(\R)$ as $\varepsilon\to0^+$ by~\cite{L23}*{Th.~6.65} whenever $(\varrho_\varepsilon)_{\varepsilon>0}\subseteq C^\infty_0(\R)$ is a family of standard non-negative mollifiers.
\end{proof}

As an immediate consequence of \cref{res:hs1,res:bambi,res:lambada}, we get the following stability result, whose simple proof is omitted.

\begin{corollary}
\label{res:limh}
Given $N\ge1$ and $s\in(0,1)$, it holds that
\begin{equation*}
\lim_{p\to1^+}
\mathfrak h_{s,p}(\mathbb H^N_+)
=
\mathfrak h_{s,1}(\mathbb H^N_+).
\end{equation*}
\end{corollary}

\subsection{Proof of \texorpdfstring{\cref{res:homo}}{Theorem 1.2}}

We begin with the case $N=1$ in \cref{res:homo}.

\begin{lemma}
\label{res:segment}
It holds that $\mathfrak h_{s,1}(\mathbb B^1)=\frac{2^{2-s}}{s}$ for all $s\in(0,1)$.
Moreover, $\mathfrak h_{s,1}(\mathbb B^1)$ is attained in $\mathcal W^{s,1}_0(\mathbb B^1)$ for all $s\in(0,1)$.
\end{lemma}

\begin{proof}
Owing to~\eqref{eq:persseg}, \eqref{eq:volsseg} and \cref{res:cheeger_omega}, we can estimate
\begin{equation*}
\mathfrak h_{s,1}(\mathbb B^1)
\le 
\frac{P_s(\mathbb B^1)}{V_{s,\mathbb B^1}(\mathbb B^1)}
=
\frac{2^{2-s}}{s}
\end{equation*} 
for all $s\in(0,1)$. 
To prove the converse inequality, 
let $E\subseteq\mathbb B^1$ be such that $|E|>0$ and define $E^\star=\left(-\frac{|E|}{2},\frac{|E|}{2}\right)\subseteq\mathbb B^1$ as in~\eqref{eq:ballification}.
By~\eqref{eq:fraciso}, we have that
\begin{equation}
\label{eq:scacco}
P_s(E) 
\ge 
P_s(E^\star) 
= 
P_s((0,1))\,|E|^{1-s}
=
\frac{4}{s(1-s)}\,|E|^{1-s}
\end{equation} 
for all $s\in(0,1)$.
On the other hand, by~\eqref{eq:hl} and \cref{res:rear}, we have that 
\begin{equation}
\label{eq:matto}
V_{s,\mathbb B^1}(E)
=
\int_{\R}
\mathbf{1}_E(x)\,\delta_s(x)\di x
\le 
\int_{\R}
\mathbf{1}_{E^\star}(x)\,\delta_s^\star(x)\di x
=
\int_{-\frac{|E|}2}^{\frac{|E|}2}\frac{\di x}{|x|^s}
=
\frac{2^s}{(1-s)}\,|E|^{1-s}
.
\end{equation}
By combining~\eqref{eq:scacco} and~\eqref{eq:matto}, we thus get that 
\begin{equation*}
\frac{P_s(E)}{V_{s,\mathbb B^1}(E)}
\ge 
\frac{4\,|E|^{1-s}}{s(1-s)}
\,
\frac{(1-s)}{2^s\,|E|^{1-s}}
=
\frac{2^{2-s}}{s}
\end{equation*}
for all $s\in(0,1)$ whenever $E\subseteq\mathbb B^1$ is such that $|E|>0$, readily yielding the equality $\mathfrak h_{s,1}(\mathbb B^1)=\frac{2^{2-s}}{s}$ in virtue of \cref{res:cheeger_relax}.
To conclude, again by \cref{res:cheeger_relax}, it is enough to check that $\mathbf 1_{\mathbb B^1}\in\mathcal W^{s,1}_0(\mathbb B^1_+)$.
As in the last part of the proof of \cref{res:hs1}, this simply follows by observing that 
\begin{equation*}
\lim_{t\to0^+}
[\mathbf 1_{(-1+t,1-t)}-\mathbf 1_{\mathbb B^1}]_{W^{s,1}(\R)}
= 
\lim_{t\to0^+}
P_s((-1,-1+t)\cup(1-t,1))=0
\end{equation*}
by~\eqref{eq:submodularity} and~\eqref{eq:persseg} and that $\mathbf 1_{(-1+t,1-t)}\in\mathcal W^{s,1}_0(\mathbb B^1)$ for every $t\in(0,1)$ by a plain convolution argument. The proof is complete.
\end{proof}

\begin{remark}
\label{rem:lucky_1d}
The strategy of the proof of \cref{res:segment} is ineffectual in the higher dimensional case $N\ge2$.
Indeed, by applying \cref{rem:rear2} below, we analogously get that
\begin{equation}
\label{eq:badluck}
\frac{P_s(E)}{V_{s,\mathbb B^N}(E)}
\ge
\frac{P_s(\mathbb B^N)}{|\mathbb B^N|^{1-\frac sN}}
\,
\frac{|E|^{1-\frac sN}}{\displaystyle\int_{E^\star}\left(1-\left(1-|x|^N\right)^\frac1N\right)^{-s}\di x}
\end{equation}
whenever $E\subseteq\mathbb B^N$ is such that $|E|>0$, where $E^\star$ is as in~\eqref{eq:ballification}.
However, we have that
\begin{equation*}
\inf_{R\in(0,1]}
\frac{|B_R|^{1-\frac sN}}{\displaystyle\int_{B_R}\left(1-\left(1-|x|^N\right)^\frac1N\right)^{-s}\di x}
=0,
\end{equation*}
so that~\eqref{eq:badluck} leads to the trivial lower bound $\mathfrak h_{s,1}(\mathbb B^N)\ge0$.
\end{remark}

We now pass to the case $N\ge2$ in \cref{res:homo}.

\begin{proof}[Proof of \cref{res:homo}]
Thanks to \cref{res:segment}, we can assume $N\ge2$.
Let 
\begin{equation*}
r_\Omega=\sup_{x\in\Omega}d_\Omega(x)\in(0,\infty)
\end{equation*}
be the \emph{inradius} of $\Omega$.
By~\cite{L16}*{Th.~1.2} (see also~\cite{L20}), we have that 
\begin{equation}
\label{eq:larson}
P(\Omega_t)
\ge 
\left(1-\frac{t}{r_\Omega}\right)^{N-1}P(\Omega)
\quad
\text{for all}\ t\in(0,r_\Omega),
\end{equation}
where 
\begin{equation*}
\Omega_t=\{x\in\Omega : d_\Omega(x)>t\}.
\end{equation*}
By scale invariance of~\eqref{eq:h_sp}, we can assume that $r_\Omega=1$ without loss of generality. 
Since $d_\Omega$ is $1$-Lipschitz, $|\nabla d_\Omega|\le 1$ a.e.\ in $\Omega$ and thus, by the usual coarea formula, we have that
\begin{equation*}
\begin{split}
V_{s,\Omega}(\Omega)
=
\int_{\Omega}\frac{1}{d_\Omega^s}\di x
\ge 
\int_{\Omega}\frac{|\nabla d_\Omega|}{d_\Omega^s}\di x
=
\int_0^1
\int_{\{d_\Omega=t\}}
\frac{1}{t^s}\di\mathscr H^{N-1}\di t
=
\int_0^1\frac{P(\Omega_t)}{t^s}\di t,
\end{split}
\end{equation*}
since 
\begin{equation*}
\partial\Omega_t=\{x\in\Omega : d_\Omega(x)=t\}
\quad
\text{and}
\quad
P(\Omega_t)=\mathscr H^{N-1}(\partial\Omega_t)
\quad
\text{for all}\
t\in(0,1).
\end{equation*}
Owing to~\eqref{eq:larson}, we hence get that
\begin{equation*}
V_{s,\Omega}(\Omega)
\ge
\int_0^1\frac{P(\Omega_t)}{t^s}\di t
\ge 
P(\Omega)\int_0^1t^{-s}\,(1-t)^{N-1}\di t
=
P(\Omega)
\,
\operatorname{B}(N,1-s).
\end{equation*} 
By \cref{res:cheeger_omega} and~\eqref{eq:betag}, we can thus estimate
\begin{equation}
\label{eq:baggio}
\mathfrak h_{s,1}(\Omega)
\le
\frac{P_s(\Omega)}{V_{s,\Omega}(\Omega)}
\le 
\frac{(1-s)\,P_s(\Omega)}{P(\Omega)}
\,
\frac{\Gamma(N+1-s)}{\Gamma(N)\,\Gamma(2-s)}
\end{equation}
and so, thanks to~\eqref{eq:hs1_H} and \eqref{eq:bambi}, we get that
\begin{equation*}
\begin{split}
\frac{\mathfrak h_{s,1}(\Omega)}{\mathfrak h_{s,1}(\mathbb H^N_+)}
\le 
\frac{(1-s)\,P_s(\Omega)}{P(\Omega)}
\,
\frac{\Gamma(N+1-s)}{\Gamma(N)\,\Gamma(2-s)}
\,
\frac{s\,\Gamma\left(\frac{N+s}2\right)}{4\,\pi^{\frac{N-1}2}\,\Gamma\left(\frac{s+1}2\right)}
\end{split}
\end{equation*}
for all $s\in(0,1)$.
Hence, by~\eqref{eq:davila} and~\eqref{eq:ballo}, we infer that
\begin{equation*}
\begin{split}
\limsup_{s\to1^-}
\frac{\mathfrak h_{s,1}(\Omega)}{\mathfrak h_{s,1}(\mathbb H^N_+)}
&\le
\lim_{s\to1^-}
\frac{(1-s)\,P_s(\Omega)}{P(\Omega)}
\,
\frac{\Gamma(N+1-s)}{\Gamma(N)\,\Gamma(2-s)}
\,
\frac{s\,\Gamma\left(\frac{N+s}2\right)}{4\,\pi^{\frac{N-1}2}\,\Gamma\left(\frac{s+1}2\right)}
\\
&=
2\,\omega_{N-1}
\,
\frac{\Gamma(N)}{\Gamma(N)\,\Gamma(1)}
\,
\frac{\Gamma\left(\frac{N+1}2\right)}{4\,\pi^{\frac{N-1}2}\,\Gamma\left(1\right)}
=
\frac{1}{2}
\end{split}
\end{equation*}
and the conclusion immediately follows by the first inequality in~\eqref{eq:hs1}.
\end{proof}

\begin{remark}
\label{rem:limits}
Given $N\ge2$ and a (non-empty) bounded open convex set $\Omega\subsetneq\R^N$, from~\eqref{eq:baggio} and~\eqref{eq:davila} we get that 
\begin{equation*}
\limsup_{s\to1^-}
\mathfrak h_{s,1}(\Omega)
\le
\lim_{s\to1^-}
\frac{(1-s)\,P_s(\Omega)}{P(\Omega)}
\,
\frac{\Gamma(N+1-s)}{\Gamma(N)\,\Gamma(2-s)}
=
2\,\omega_{N-1},
\end{equation*}
while, thanks to \cref{res:hs1,res:bambi}, 
\begin{equation*}
\liminf_{s\to1^-}
\mathfrak h_{s,1}(\Omega)
\ge 
\frac12
\lim_{s\to1^-}
\mathfrak h_{s,1}(\mathbb H^N_+)
=
2\,C_{N,1}
=
2\,\omega_{N-1}.
\end{equation*} 
This implies that~\eqref{eq:homo_lim} can be equivalently reformulated as
\begin{equation*}
\lim_{s\to1^-}
\mathfrak h_{s,1}(\Omega)
=
2\,\omega_{N-1}
<
4\,\omega_{N-1}
=
\lim_{s\to1^-}
\mathfrak h_{s,1}(\mathbb H^N_+).
\end{equation*}
Noteworthy, thanks to~\cite{MS02}*{Th.~3} (see also~\cite{MS03}), \cref{res:cheeger_omega,res:bambi,res:hs1}, we also have that 
\begin{equation*}
\limsup_{s\to0^+}
s\,\mathfrak h_{s,1}(\Omega)
\le 
\lim_{s\to0^+}
\frac{s\,P_s(\Omega)}{V_{s,\Omega}(\Omega)}
=
\frac{2\,N\,\omega_N\,|\Omega|}{|\Omega|}
=
2\,N\,\omega_N
=
4\,C_{N,0}
=
\lim_{s\to0^+}
s\,\mathfrak h_{s,1}(\mathbb H^N_+).
\end{equation*}
Due to~\eqref{eq:homo_1}, in the case $N=1$, the above improves to \begin{equation*}
\lim_{s\to0^+}
s\,\mathfrak h_{s,1}(\Omega)
=
\lim_{s\to0^+}
s\,\mathfrak h_{s,1}(\mathbb H^1_+),
\end{equation*} 
but we do not know if this is also the case for $N\ge2$.
\end{remark}

As already mentioned in the introduction, the limit in~\eqref{eq:homo_lim} can be slightly improved if $\Omega$ is an open ball by exploiting the main result of~\cite{G20}, as follows.

\begin{proposition}
\label{res:ball}
If $N\ge2$, then
\begin{equation}
\label{eq:ball_2}
\mathfrak h_{s,1}(\mathbb B^N)
\le
\frac{\pi^{\frac12}}{\Gamma\left(1-\frac s2\right)\,\Gamma\left(\frac12+\frac s2\right)}
\,
\frac{\Gamma\left(\frac{N+s}{2}\right)}{\Gamma\left(\frac{N-s}{2}\right)}
\,
\frac{\Gamma(N-s)}{\Gamma(N)}
\,
\mathfrak h_{s,1}(\mathbb H^N_+)
\end{equation}
for all $s\in(0,1)$.
\end{proposition}

\begin{proof}[Proof of \cref{res:ball}]
Since $d_{\mathbb B^N}(x)=1-|x|$ for all $x\in\mathbb B^N$, we can compute
\begin{equation*}
\begin{split}
V_{s,\mathbb B^N}(\mathbb B^N)
&=
\int_{\mathbb B^N}(1-|x|)^{-s}\di x
=
N\omega_N
\,
\mathrm B(N,1-s)
=
N\omega_N
\,
\frac{\Gamma(N)\,\Gamma(1-s)}{\Gamma(N+1-s)}
\end{split}
\end{equation*}
for all $s\in(0,1)$. 
Therefore, by \cref{res:cheeger_omega} and recalling~\eqref{eq:garofalo} and~\eqref{eq:ballo}, we get that
\begin{equation*}
\begin{split}
\mathfrak h_{s,1}(\mathbb B^N)
&\le
\frac{P_s(\mathbb B^N)}{V_{s,\mathbb B^N}(\mathbb B^N)}
\\
&=
\omega_N^{1-\frac sN}
\,
\frac{N\,\pi^{\frac{N+s}2}\,\Gamma(1-s)}{\frac s2\,\Gamma\left(\frac N2+1\right)^{\frac{s}N}\,\Gamma\left(1-\frac s2\right)\,\Gamma\left(\frac{N-s}{2}+1\right)} 
\,
\frac{\Gamma(N+1-s)}{N\omega_N\,\Gamma(N)\,\Gamma(1-s)}
\\
&=
\omega_N^{-\frac sN}
\,
\frac{\pi^{\frac{N+s}2}}{\frac s2\,\Gamma\left(\frac N2+1\right)^{\frac{s}N}\,\Gamma\left(1-\frac s2\right)\,\Gamma\left(\frac{N-s}{2}+1\right)}
\,
\frac{\Gamma(N+1-s)}{\Gamma(N)}
\\
&=
\frac{2\,\pi^{\frac{N}2}\,\Gamma(N+1-s)}{s\,\Gamma(N)\,\Gamma\left(1-\frac s2\right)\,\Gamma\left(\frac{N-s}{2}+1\right)}
\end{split}
\end{equation*}
for all $s\in(0,1)$.
We can thus estimate
\begin{equation*}
\begin{split}
\frac{\mathfrak h_{s,1}(\mathbb B^N)}{\mathfrak h_{s,1}(\mathbb H^N_+)}
&\le
\frac{2\,\pi^{\frac{N}2}\,\Gamma(N+1-s)}{s\,\Gamma(N)\,\Gamma\left(1-\frac s2\right)\,\Gamma\left(\frac{N-s}{2}+1\right)}
\,
\frac{s\,\Gamma\left(\frac{N+s}2\right)}{4\,\pi^{\frac{N-1}2}\,\Gamma\left(\frac{s+1}2\right)}
\\
&=
\frac{\pi^{\frac12}}{\Gamma\left(1-\frac s2\right)\,\Gamma\left(\frac12+\frac s2\right)}
\,
\frac{\Gamma\left(\frac{N+s}{2}\right)}{\Gamma\left(\frac{N-s}{2}\right)}
\,
\frac{\Gamma(N-s)}{\Gamma(N)}
\end{split}
\end{equation*}
for all $s\in(0,1)$, proving~\eqref{eq:ball_2} and concluding the proof.
\end{proof}

\subsection{Proof of \texorpdfstring{\cref{res:salto}}{Corollary 1.3}}

We can now study the sharp constant in~\eqref{eq:h_sp} for $s\in(0,1)$ and $p\in(1,\infty)$ such that $sp<1$.

\begin{proof}[Proof of \cref{res:salto}]
By~\eqref{eq:frachardy_estim}, \cref{res:hs1}, \eqref{eq:stab_h} and \cref{res:limh}, we get that
\begin{equation*}
\frac12
\le 
\limsup_{p\to1^+}
\frac{\mathfrak h_{s,p}(\Omega)}{\mathfrak h_{s,p}(\mathbb H^N_+)}
\le 
\frac{\mathfrak h_{s,1}(\Omega)}{\mathfrak h_{s,1}(\mathbb H^N_+)}
\end{equation*}
for all $s\in(0,1)$.
The conclusion hence plainly follows from \cref{res:homo}.
\end{proof}

\subsection{Proof of \texorpdfstring{\cref{res:fs}}{Corollary 1.5}}

We now pass to the study of the non-convex case $\Omega=\R^N\setminus\{0\}$, (re)proving~\cite{FS08}*{Th.~1.1} in the case $p=1$.

\begin{proof}[Proof of \cref{res:fs}]
Let us set $\mathbb U^N=\R^N\setminus\{0\}$ for brevity.
Since 
\begin{equation*}
V_{s,\mathbb U^N}(\mathbb B^N)
=
\int_{\mathbb B^N}\frac{\di x}{|x|^s}
=
N\,\omega_N
\int_0^1 r^{N-1-s}\di r
=
\frac{N\,\omega_N}{N-s},
\end{equation*}
by \cref{res:cheeger_s-neg} (applied with $U=\R^N$ and $F=\{0\}$) we can thus estimate
\begin{equation}
\label{eq:culo}
\mathfrak g_s(\mathbb U^N)
\le 
\frac{P_s(\mathbb B^N)}{V_{s,\mathbb U^N}(\mathbb B^N)}
=
P_s(\mathbb B^N)
\,
\frac{N-s}{N\,\omega_N}
\end{equation}
for all $s\in(0,1)$.
On the other hand, given a measurable set $E\subseteq\R^N$ such that $|E|\in(0,\infty)$, for all $s\in(0,1)$, by~\eqref{eq:fraciso}, we can estimate
\begin{equation*}
P_s(E)
\ge 
\frac{P_s(\mathbb B^N)}{|\mathbb B^N|^{1-\frac sN}}
\,
|E|^{1-\frac sN}
=
\frac{P_s(\mathbb B^N)}{\omega_N^{1-\frac sN}}
\,
|E|^{1-\frac sN}
,
\end{equation*}
while, by~\eqref{eq:hl} and owing to the fact that $x\mapsto|x|^{-s}$ is radially symmetric decreasing,
\begin{equation*}
V_{s,\mathbb U^N}(E)
=
\int_E
\frac{1}{|x|^s}
\di x
\le 
\int_{E^\star}
\frac{1}{|x|^s}
\di x
=
\frac{N\,\omega_N}{N-s}\,\left(\frac{|E|}{\omega_N}\right)^{1-\frac{s}{N}}
\end{equation*}
where $E^*$ is as in~\eqref{eq:ballification}.
Thus, we get that
\begin{equation}
\label{eq:camicia}
\frac{P_s(E)}{V_{s,\mathbb U^N}(E)}
\ge 
\frac{P_s(\mathbb B^N)}{\omega_N^{1-\frac sN}}
\,
|E|^{1-\frac sN}
\,
\frac{N-s}{N\,\omega_N}
\,
\left(\frac{\omega_N}{|E|}\right)^{1-\frac{s}{N}}
=
P_s(\mathbb B^N)
\,
\frac{N-s}{N\,\omega_N}\end{equation}
for all $s\in(0,1)$, whenever $E\subseteq\R^N$ is such that $|E|\in(0,\infty)$. 
By combining~\eqref{eq:culo} and~\eqref{eq:camicia}, by~\cite{G20}*{Prop.~1.1} (see~\eqref{eq:garofalo}), and recalling~\eqref{eq:ballo}, we conclude that 
\begin{equation*}
\mathfrak g_s(\mathbb U^N)
=
P_s(\mathbb B^N)
\,
\frac{N-s}{N\,\omega_N}
=
\frac{4}{s}
\,
\frac{\pi^{\frac N2}\,\Gamma(1-s)}{\Gamma\left(\frac{N-s}{2}\right)\,\Gamma\left(1-\frac s2\right)}
\end{equation*}
for all $s\in(0,1)$, which yields~\eqref{eq:fs_c} in virtue of \cref{res:cheeger}.
The characterization of the equality cases in~\eqref{eq:fs_h} follows as in~\cite{FS08}, so we omit the details.
\end{proof}

\subsection{Proof of \texorpdfstring{\cref{res:onen}}{Theorem 1.6}}

For the proof of \cref{res:onen}, we need to establish the following simple preliminary result.

\begin{lemma}
\label{res:equivint}
Let $N=1$ and $\Omega\subsetneq\R$ be the union of $n\in\mathbb N$ (non-empty) disjoint bounded open intervals $\{I_k:k=1,\dots,n\}$.
If $\Omega$ is equivalent to an interval, then
\begin{equation}
\label{eq:equivint}
\mathfrak h_{s,1}(\Omega)
\le 
\frac{\left(\sum_{k=1}^n|I_k|\right)^{1-s}}{\sum_{k=1}^n|I_k|^{1-s}}
\,
\mathfrak h_{s,1}(\R\setminus\{0\})
\end{equation}
for every $s\in(0,1)$.
If, in addition, 
\begin{equation}
\label{eq:equivinteq}
\mathfrak h_{s,1}(\Omega)
\ge
\frac{\mathfrak h_{s,1}(\R\setminus\{0\})}{n^s}
\end{equation}
then all the $I_k$'s have the same measure.  
\end{lemma}

\begin{proof}
Since $\Omega$ is equivalent to an interval, by~\eqref{eq:persseg},  we have 
\begin{equation*}
P_s(\Omega)
=
\frac{4}{s(1-s)}
\,
|\Omega|^{1-s}
=
\frac{4}{s(1-s)}
\,
\left(\sum_{k=1}^n|I_k|\right)^{1-s}
\end{equation*}
while, since the $I_k$'s are pairwise disjoint, by~\eqref{eq:volsseg} we can write
\begin{equation*}
V_{s,\Omega}(\Omega)
=
\sum_{k=1}^nV_{s,\Omega}(I_k)
=
\sum_{k=1}^nV_{s,I_k}(I_k)
=
\frac{2^s}{1-s}
\,
\sum_{k=1}^n|I_k|^{1-s}
\end{equation*}
for all $s\in(0,1)$. Therefore, inequality~\eqref{eq:equivint} follows from \cref{res:cheeger_relax} and~\eqref{eq:fs1}.
Finally, if~\eqref{eq:equivinteq} additionally holds, then from~\eqref{eq:equivint} we deduce that
\begin{equation*}
\frac1n\sum_{k=1}^n|I_k|^{1-s}
\le 
\left(\frac1n\sum_{k=1}^n|I_k|\right)^{1-s}
\end{equation*}
which implies that the $|I_k|$'s are all equal, since $t\mapsto t^{1-s}$ is strictly concave on~$(0,\infty)$.
\end{proof}

\begin{proof}[Proof of \cref{res:onen}]
We split the proof into two parts.

\vspace{1ex}

\textit{Proof of~\eqref{eq:onen}}.
By combining~\eqref{eq:posze} with~\eqref{eq:aceto} in \cref{res:nomega_vol}, we get that
\begin{equation*}
\frac{[u]_{W^{s,p}(\R)}^p}{\displaystyle\int_{\Omega}\frac{|u|^p}{d_\Omega^{sp}}\di x}
\ge 
\frac{[u^\star]_{W^{s,p}(\R)}^p}{\displaystyle n^{sp}
\int_{\R}
\frac{(u^\star)^p}{|x|^{sp}}
\di x}
\ge 
\frac{\mathfrak h_{s,p}\left(\R^N\setminus\{0\}\right)}{n^{sp}}
\end{equation*}
owing to \cref{res:fs}, from which we  deduce~\eqref{eq:onen}.
Moreover, if $p=1$ and $\Omega$ is equivalent to an interval, then \cref{res:equivint} proves that~\eqref{eq:onen} holds as an equality if and only if all the $I_k$'s have the same measure.

\vspace{1ex}

\textit{Proof of~\eqref{eq:rmenz}}.
By \cref{res:spacco,res:cheeger_relax}, for any $R\ge1$, we can estimate
\begin{equation}
\label{eq:caffe}
\mathfrak h_{s,p}((-R,R)\setminus\mathbb Z)
\le 
\mathfrak h_{sp,1}((-R,R)\setminus\mathbb Z)
=
\mathfrak g_{sp}((-R,R)\setminus\mathbb Z).
\end{equation}
Now, since $|\mathbb Z|=0$, we have that
\begin{equation}
\label{eq:latte}
P_{sp}((-R,R)\setminus\mathbb Z)
=
P_{sp}((-R,R))
=
(2R)^{1-sp}\,\frac{4}{sp\,(1-sp)},
\end{equation}
thanks to~\eqref{eq:persseg} (with $sp<1$ in place of $s$), and also
\begin{equation}
\label{eq:macchiato}
\begin{split}
V_{sp,(-R,R)\setminus\mathbb Z}((-R,R)\setminus\mathbb Z)
&\ge 
V_{sp,(-R,R)\setminus\mathbb Z}((-\lfloor R\rfloor,\lfloor R\rfloor))
\\
&=
V_{sp,(-\lfloor R\rfloor,\lfloor R\rfloor)\setminus\mathbb Z}((-\lfloor R\rfloor,\lfloor R\rfloor))
\\
&=
V_{sp,(-\lfloor R\rfloor,\lfloor R\rfloor)}((-\lfloor R\rfloor,\lfloor R\rfloor))
=
\lfloor R\rfloor
\,
\frac{2^{1+sp}}{1-sp}
\end{split}
\end{equation}
thanks to~\eqref{eq:volsseg}. 
Hence, by combining~\eqref{eq:caffe}, \eqref{eq:latte} and~\eqref{eq:macchiato}, we infer that 
\begin{equation*}
\mathfrak h_{s,p}((-R,R)\setminus\mathbb Z)
\le 
\mathfrak g_{sp}((-R,R)\setminus\mathbb Z)
\le
\frac{P_{sp}((-R,R)\setminus\mathbb Z)}{V_{sp,(-R,R)\setminus\mathbb Z}((-R,R)\setminus\mathbb Z)}
=
\frac{R^{1-sp}}{\lfloor R\rfloor}
\,
\frac{4^{1-sp}}{sp}
\end{equation*}
proving~\eqref{eq:rmenz} for any $R\ge1$. 
For the limit case $R=\infty$, we can analogously bound 
\begin{equation*}
\mathfrak h_{s,p}(\R\setminus\mathbb Z)
\le 
\mathfrak h_{sp,1}(\R\setminus\mathbb Z)
=
\mathfrak g_{sp}(\R\setminus\mathbb Z)
\le 
\mathfrak g_{sp}((-m,m)\setminus\mathbb Z)
\le 
\frac{1}{m^{sp}}
\,
\frac{4^{1-sp}}{sp}
\end{equation*}
for all $m\in\mathbb N$, because $V_{s,\R\setminus\mathbb Z}(E)=V_{s,(-m,m)\setminus\mathbb Z}(E)$ for any $E\subset(-m,m)\setminus\mathbb Z$.
We hence get that $\mathfrak h_{s,p}(\R\setminus\mathbb Z)
=0$, concluding the proof. 
\end{proof}

\subsection{Proof of \texorpdfstring{\cref{res:nsegments}}{Theorem 1.7}}

We now turn to the proof of  \cref{res:nsegments}---actually, of the following more precise result.
Here and below, given an non-empty set $A\subseteq\R$ and $\delta\in(0,\infty)$, we let $A^\delta=\{x\in\R:\operatorname{dist}(x,A)<\delta\}$ be the \emph{open $\delta$-neighborhood} of~$A$.

\begin{theorem}
\label{res:nsegments_delta}
Let $\Omega\subsetneq\R$, $n\in\mathbb N$ and $\ell,\delta\in(0,\infty)$ be as in \cref{res:nsegments}.
For $s\in(0,1)$ and $p\in[1,\infty)$ such that $sp\le1$, it holds that 
\begin{equation*}
[u]_{W^{s,p}(\Omega^\delta)}^p
\ge 
\frac{2}{sp}
\left(
1-
\left(\frac{\ell}{\ell+\delta}\right)^{sp}
\,
\right)
\int_{\Omega}
\frac{|u|^p}{d_\Omega^{sp}}\di x
\end{equation*}
for all $u\in C^\infty_0(\Omega)$.
\end{theorem}

To prove \cref{res:nsegments_delta}, we need the following preliminary result; that is, \cref{res:nsegments_delta} in the case of a single interval.

\begin{proposition}
\label{res:intfat}Let $I\subsetneq\R$ be a non-empty bounded open interval with $|I|=\ell\in(0,\infty)$.
For $s\in(0,1)$ and $p\in[1,\infty)$  such that $sp\le1$, it holds that
\begin{equation*}
[u]_{W^{s,p}(I^\delta)}^p
\ge
\frac{2^{2-sp}}{sp}
\left(1-\left(\frac{\ell}{\ell+\delta}\right)^{sp}\,\right)
\int_I\frac{|u|^p}{d_I^{sp}}\di x
\end{equation*}
for any $\delta\in(0,\infty)$ and $u\in C^\infty_0(I)$.
\end{proposition}

\begin{proof}
For convenience, let $a,b\in\R$, $a<b$, be such that $I=(a,b)$.
We let $\varphi\in\operatorname{Lip}(\R)$ be a Lipschitz function such that $\operatorname{supp}\varphi\subseteq I$ and $\varphi\ge0$.
For any $\varepsilon>0$, we define 
\begin{equation}
\label{eq:mango}
\bigtriangleup_\varepsilon=\{(x,y)\in\R^2 : |x-y|\le\varepsilon\},
\end{equation}
and we let $J_p\colon\R\to\R$ be given by 
\begin{equation*}
J_p(t)
=
\begin{cases}
|t|^{p-2}\,t 
& \text{for}\ p\in(1,\infty),
\\[1ex]
\operatorname{sgn}(t)
& \text{for}\ p=1,
\end{cases}
\end{equation*}
for all $t\ne0$, and $J_p(0)=0$.
With this notation in force, we observe that 
\begin{equation*}
\begin{split}
\limsup_{\varepsilon\to0^+}
\bigg|
&
\iint_{(I^\delta\times I^\delta)\cap\bigtriangleup_\varepsilon}
\frac{J_p(\mathbf 1_I(x)-\mathbf 1_I(y))}{|x-y|^{1+sp}}\,(\varphi(x)-\varphi(y))\di x\di y
\,
\bigg|
\\
&\le 
\limsup_{\varepsilon\to0^+}
\iint_{(I^\delta\times I^\delta)\cap\bigtriangleup_\varepsilon}
\frac{|\mathbf 1_I(x)-\mathbf 1_I(y)|^{p-1}}{|x-y|^{1+sp}}\,|\varphi(x)-\varphi(y)|\di x\di y
\\
&\le 
\limsup_{\varepsilon\to0^+}
\iint_{\bigtriangleup_\varepsilon}
\frac{|\varphi(x)-\varphi(y)|}{|x-y|^{1+sp}}\,\di x\di y=0
\end{split}
\end{equation*}
by the Dominated Convergence Theorem, since $\varphi\in W^{sp,1}(\R)$.
Therefore, we can write 
\begin{equation}
\label{eq:kiwi}
\begin{split}
\iint_{I^\delta\times I^\delta}
&
\frac{J_p(\mathbf 1_I(x)-\mathbf 1_I(y))}{|x-y|^{1+sp}}\,(\varphi(x)-\varphi(y))\di x\di y
\\
&=
\lim_{\varepsilon\to0^+}
\iint_{(I^\delta\times I^\delta)\setminus\bigtriangleup_\varepsilon}
\frac{J_p(\mathbf 1_I(x)-\mathbf 1_I(y))}{|x-y|^{1+sp}}\,(\varphi(x)-\varphi(y))\di x\di y.
\end{split}
\end{equation}
Now, given $\varepsilon>0$, we can rewrite
\begin{equation}
\label{eq:ananas}
\begin{split}
\iint_{(I^\delta\times I^\delta)\setminus\bigtriangleup_\varepsilon}
&
\frac{J_p(\mathbf 1_I(x)-\mathbf 1_I(y))}{|x-y|^{1+sp}}\,(\varphi(x)-\varphi(y))\di x\di y
\\
&=
2\iint_{(I^\delta\times I^\delta)\setminus\bigtriangleup_\varepsilon}
\frac{J_p(\mathbf 1_I(x)-\mathbf 1_I(y))}{|x-y|^{1+sp}}\,\varphi(x)\di x\di y
\\
&=
2\int_I\varphi(x)
\int_{\{y\in I^\delta\setminus I\, :\, |y-x|>\varepsilon\}}\frac{\di y}{|y-x|^{1+sp}}
\di x.
\end{split}
\end{equation}
Observing that
\begin{equation*}
\frac{1}{(x-a+\delta)^{sp}}
\le 
\frac{C_\delta^{sp}}{(x-a)^{sp}}
\quad
\text{and}
\quad
\frac{1}{(b+\delta-x)^{sp}}
\le 
\frac{C_\delta^{sp}}{(b-x)^{sp}}
\end{equation*}
for all $x\in I$,
where
\begin{equation*}
C_\delta=\frac{b-a}{b-a+\delta}\in(0,1),
\end{equation*}
we can estimate
\begin{equation}
\label{eq:mandarino}
\begin{split}
\lim_{\varepsilon\to0^+}
&
\int_{\{y\in I^\delta\setminus I : |y-x|>\varepsilon\}}\frac{\di y}{|y-x|^{1+sp}}
\\
&=
\frac{1}{sp}
\left(
\frac{1}{(x-a)^{sp}}
-
\frac{1}{(x-a+\delta)^{sp}}
+\frac{1}{(b-x)^{sp}}
-
\frac{1}{(b+\delta-x)^{sp}}
\right)
\\
&\ge 
\frac{(1-C_\delta^{sp})}{sp}
\left(
\frac{1}{(x-a)^{sp}}
+
\frac{1}{(b-x)^{sp}}
\right)
\end{split}
\end{equation}
for all $x\in I$.
Since $sp\le1$, for each $x\in I$, we further have that
\begin{equation}
\label{eq:mandarancio}
\frac{1}{(x-a)^{sp}}
+
\frac{1}{(b-x)^{sp}}
=
\frac{(x-a)^{sp}+(b-x)^{sp}}{(x-a)^{sp}\,(b-x)^{sp}}
\ge
\frac{2^{1-sp}\,(b-a)^{sp}}{(x-a)^{sp}\,(b-x)^{sp}}
\ge 
\frac{2^{1-sp}}{d_I(x)^{sp}} 
\end{equation}
and thus, owing to~\eqref{eq:kiwi}, \eqref{eq:ananas}, Fatou's Lemma (recall that $\varphi\ge0$), \eqref{eq:mandarino} and~\eqref{eq:mandarancio}, 
\begin{equation}
\label{eq:perospino}
\begin{split}
\iint_{I^\delta\times I^\delta}
&
\frac{J_p(\mathbf 1_I(x)-\mathbf 1_I(y))}{|x-y|^{1+sp}}\,(\varphi(x)-\varphi(y))\di x\di y
\\
&=
\lim_{\varepsilon\to0^+}
\iint_{(I^\delta\times I^\delta)\setminus\bigtriangleup_\varepsilon}
\frac{J_p(\mathbf 1_I(x)-\mathbf 1_I(y))}{|x-y|^{1+sp}}\,(\varphi(x)-\varphi(y))\di x\di y
\\
&=
\lim_{\varepsilon\to0^+}
2\int_I\varphi(x)
\int_{\{y\in I^\delta\setminus I\, :\, |y-x|>\varepsilon\}}\frac{\di y}{|y-x|^{1+sp}}
\di x
\ge 
\frac{2^{2-sp}\,(1-C_\delta^{sp})}{sp}
\int_I\frac{\varphi}{d_I^{sp}}\di x.
\end{split}
\end{equation}
We now choose the Lipschitz function 
\begin{equation*}
\varphi=\frac{|u|^p}{(\varepsilon+\mathbf 1_I)^{p-1}},
\end{equation*}
where $u\in C^\infty_0(I)$.
Noticing that, by~\cite{BBZ23}*{Lem.~2.4},
\begin{equation*}
J_p(\mathbf 1_I(x)-\mathbf 1_I(y))
\left(
\frac{|u(x)|^p}{(\varepsilon+\mathbf 1_I(x))^{p-1}}
-
\frac{|u(y)|^p}{(\varepsilon+\mathbf 1_I(y))^{p-1}}
\right)
\le
|u(x)-u(y)|^p
\end{equation*}
for all $x,y\in I^\delta$ and $\varepsilon>0$, owing to~\eqref{eq:perospino}, we conclude that
\begin{equation*}
\begin{split}
[u]_{W^{s,p}(I^\delta)}^p
&\ge 
\liminf_{\varepsilon\to0^+}
\iint_{I^\delta\times I^\delta}
\frac{J_p(\mathbf 1_I(x)-\mathbf 1_I(y))}{|x-y|^{1+sp}}
\left(
\frac{|u(x)|^p}{(\varepsilon+\mathbf 1_I(x))^{p-1}}
-
\frac{|u(y)|^p}{(\varepsilon+\mathbf 1_I(y))^{p-1}}
\right)
\di x\di y
\\
&\ge
\liminf_{\varepsilon\to0^+}
\frac{2^{2-sp}\,(1-C_\delta^{sp})}{sp}
\int_I\frac{|u|^p}{d_I^{sp}(\varepsilon+\mathbf 1_I)^{p-1}}\di x
\ge
\frac{2^{2-sp}\,(1-C_\delta^{sp})}{sp}
\int_I\frac{|u|^p}{d_I^{sp}}\di x,
\end{split}
\end{equation*}
and the proof is complete.
\end{proof}

We can now prove \cref{res:nsegments_delta}, and thus \cref{res:nsegments}, by exploiting \cref{res:intfat}.

\begin{proof}[Proof of \cref{res:nsegments_delta}]
Let $I_i$, $i=1,\dots,n$, be the $n$ intervals composing $\Omega$; that is,   $\Omega=\bigcup_{i=1}^nI_i$ with disjoint union.
By assumption, the $\delta$-enlarged intervals $I_i^\delta$, $i=1,\dots,n$, are pairwise disjoint. 
Thus, thanks to \cref{res:intfat}, we can estimate
\begin{equation*}
\begin{split}
[u]_{W^{s,p}(\Omega^\delta)}^p
&\ge 
\sum_{i=1}^n
[u]_{W^{s,p}(I^\delta_i)}^p
\ge
\sum_{i=1}^n
\frac{2^{2-sp}}{sp}
\left(1-\left(\frac{|I_i|}{|I_i|+\delta}\right)^{sp}\,\right)
\int_{I_i}\frac{|u|^p}{d_{I_i}^{sp}}\di x
\\
&\ge
\frac{2^{2-sp}}{sp}
\left(1-\left(\frac{\ell}{\ell+\delta}\right)^{sp}\,\right)
\sum_{i=1}^n
\int_{I_i}\frac{|u|^p}{d_{I_i}^{sp}}\di x
=
\frac{2^{2-sp}}{sp}
\left(1-\left(\frac{\ell}{\ell+\delta}\right)^{sp}\,\right)
\int_{\Omega}\frac{|u|^p}{d_{\Omega}^{sp}}\di x
\end{split}
\end{equation*}
whenever $u\in C^\infty_0(\Omega)$, concluding the proof.
\end{proof}

\begin{remark}
Concerning the constant appearing in \cref{res:nsegments_delta,res:intfat}, we observe that
\begin{equation*}
\lim_{\delta\to0^+}
\frac{2^{2-sp}}{sp}
\left(1-\left(\frac{\ell}{\ell+\delta}\right)^{sp}\,\right)
=0,
\end{equation*} 
in accordance with the fact that the \emph{regional} fractional Hardy inequality~\eqref{eq:frac_regional} cannot hold for $sp\le 1$ on bounded intervals, recall \cref{rem:regional}. 
\end{remark}



\begin{bibdiv}
\begin{biblist}

\bib{AL89}{article}{
   author={Almgren, Frederick J., Jr.},
   author={Lieb, Elliott H.},
   title={Symmetric decreasing rearrangement is sometimes continuous},
   journal={J. Amer. Math. Soc.},
   volume={2},
   date={1989},
   number={4},
   pages={683--773},
   issn={0894-0347},
   review={\MR{1002633}},
   doi={10.2307/1990893},
}

\bib{BS24}{article}{
   author={Bessas, Konstantinos},
   author={Stefani, Giorgio},
   title={Non-local BV functions and a denoising model with $L^1$ fidelity},
   journal={Adv. Calc. Var.},
   volume={18},
   date={2025},
   number={1},
   pages={189--217},
   issn={1864-8258},
   review={\MR{4845992}},
   doi={10.1515/acv-2023-0082},
}
\bib{BBZ23}{article}{
   author={Bianchi, Francesca},
   author={Brasco, Lorenzo},
   author={Zagati, Anna Chiara},
   title={On the sharp Hardy inequality in Sobolev-Slobodecki\u i\ spaces},
   journal={Math. Ann.},
   volume={390},
   date={2024},
   number={1},
   pages={493--555},
   issn={0025-5831},
   review={\MR{4800921}},
   doi={10.1007/s00208-023-02770-z},
}

\bib{BD11}{article}{
   author={Bogdan, Krzysztof},
   author={Dyda, Bart\l omiej},
   title={The best constant in a fractional Hardy inequality},
   journal={Math. Nachr.},
   volume={284},
   date={2011},
   number={5-6},
   pages={629--638},
   issn={0025-584X},
   review={\MR{2663757}},
   doi={10.1002/mana.200810109},
}

\bib{BC18}{article}{
   author={Brasco, Lorenzo},
   author={Cinti, Eleonora},
   title={On fractional Hardy inequalities in convex sets},
   journal={Discrete Contin. Dyn. Syst.},
   volume={38},
   date={2018},
   number={8},
   pages={4019--4040},
   issn={1078-0947},
   review={\MR{3814363}},
   doi={10.3934/dcds.2018175},
}

\bib{BGV21}{article}{
   author={Brasco, Lorenzo},
   author={G\'{o}mez-Castro, David},
   author={V\'{a}zquez, Juan Luis},
   title={Characterisation of homogeneous fractional Sobolev spaces},
   journal={Calc. Var. Partial Differential Equations},
   volume={60},
   date={2021},
   number={2},
   pages={Paper No. 60, 40},
   issn={0944-2669},
   review={\MR{4225499}},
   doi={10.1007/s00526-021-01934-6},
}

\bib{BLP14}{article}{
   author={Brasco, L.},
   author={Lindgren, E.},
   author={Parini, E.},
   title={The fractional Cheeger problem},
   journal={Interfaces Free Bound.},
   volume={16},
   date={2014},
   number={3},
   pages={419--458},
   issn={1463-9963},
   review={\MR{3264796}},
   doi={10.4171/IFB/325},
}

\bib{CP24}{article}{
   author={Cinti, Eleonora},
   author={Prinari, Francesca},
   title={On fractional Hardy-type inequalities in general open sets},
   journal={ESAIM Control Optim. Calc. Var.},
   volume={30},
   date={2024},
   pages={Paper No. 77, 26},
   issn={1292-8119},
   review={\MR{4805757}},
   doi={10.1051/cocv/2024066},
}
\bib{D02}{article}{
   author={D\'{a}vila, J.},
   title={On an open question about functions of bounded variation},
   journal={Calc. Var. Partial Differential Equations},
   volume={15},
   date={2002},
   number={4},
   pages={519--527},
   issn={0944-2669},
   review={\MR{1942130}},
   doi={10.1007/s005260100135},
}

\bib{D04}{article}{
   author={Dyda, Bart\l omiej},
   title={A fractional order Hardy inequality},
   journal={Illinois J. Math.},
   volume={48},
   date={2004},
   number={2},
   pages={575--588},
   issn={0019-2082},
   review={\MR{2085428}},
}

\bib{DK22}{article}{
   author={Dyda, Bart\l omiej},
   author={Kijaczko, Micha\l},
   title={On density of compactly supported smooth functions in fractional
   Sobolev spaces},
   journal={Ann. Mat. Pura Appl. (4)},
   volume={201},
   date={2022},
   number={4},
   pages={1855--1867},
   issn={0373-3114},
   review={\MR{4454384}},
   doi={10.1007/s10231-021-01181-8},
}

\bib{DV14}{article}{
   author={Dyda, Bart\l omiej},
   author={V\"ah\"akangas, Antti V.},
   title={A framework for fractional Hardy inequalities},
   journal={Ann. Acad. Sci. Fenn. Math.},
   volume={39},
   date={2014},
   number={2},
   pages={675--689},
   issn={1239-629X},
   review={\MR{3237044}},
   doi={10.5186/aasfm.2014.3943},
}

\bib{DV15}{article}{
   author={Dyda, Bart\l omiej},
   author={V\"ah\"akangas, Antti V.},
   title={Characterizations for fractional Hardy inequality},
   journal={Adv. Calc. Var.},
   volume={8},
   date={2015},
   number={2},
   pages={173--182},
   issn={1864-8258},
   review={\MR{3331699}},
   doi={10.1515/acv-2013-0019},
}

\bib{EHV14}{article}{
   author={Edmunds, David E.},
   author={Hurri-Syrj\"anen, Ritva},
   author={V\"ah\"akangas, Antti V.},
   title={Fractional Hardy-type inequalities in domains with uniformly fat
   complement},
   journal={Proc. Amer. Math. Soc.},
   volume={142},
   date={2014},
   number={3},
   pages={897--907},
   issn={0002-9939},
   review={\MR{3148524}},
   doi={10.1090/S0002-9939-2013-11818-6},
}

\bib{FMT13}{article}{
   author={Filippas, Stathis},
   author={Moschini, Luisa},
   author={Tertikas, Achilles},
   title={Sharp trace Hardy-Sobolev-Maz'ya inequalities and the fractional
   Laplacian},
   journal={Arch. Ration. Mech. Anal.},
   volume={208},
   date={2013},
   number={1},
   pages={109--161},
   issn={0003-9527},
   review={\MR{3021545}},
   doi={10.1007/s00205-012-0594-4},
}

\bib{FMT18}{article}{
   author={Filippas, Stathis},
   author={Moschini, Luisa},
   author={Tertikas, Achilles},
   title={Correction to: Sharp trace Hardy-Sobolev-Maz'ya inequalities and
   the fractional Laplacian [MR3021545]},
   journal={Arch. Ration. Mech. Anal.},
   volume={229},
   date={2018},
   number={3},
   pages={1281--1286},
   issn={0003-9527},
   review={\MR{3814603}},
   doi={10.1007/s00205-018-1264-y},
}

\bib{FPSS24}{article}{
   author={Franceschi, Valentina},
   author={Pinamonti, Andrea},
   author={Saracco, Giorgio},
   author={Stefani, Giorgio},
   title={The Cheeger problem in abstract measure spaces},
   journal={J. Lond. Math. Soc. (2)},
   volume={109},
   date={2024},
   number={1},
   pages={Paper No. e12840, 55},
   issn={0024-6107},
   review={\MR{4754430}},
   doi={10.1112/jlms.12840},
}

\bib{FS10}{article}{
   author={Frank, Rupert L.},
   author={Seiringer, Robert},
   title={Sharp fractional Hardy inequalities in half-spaces},
   conference={
      title={Around the research of Vladimir Maz'ya. I},
   },
   book={
      series={Int. Math. Ser. (N. Y.)},
      volume={11},
      publisher={Springer, New York},
   },
   isbn={978-1-4419-1340-1},
   isbn={978-5-9018-7341-0},
   date={2010},
   pages={161--167},
   review={\MR{2723817}},
   doi={10.1007/978-1-4419-1341-8\_6},
}

\bib{FS08}{article}{
   author={Frank, Rupert L.},
   author={Seiringer, Robert},
   title={Non-linear ground state representations and sharp Hardy
   inequalities},
   journal={J. Funct. Anal.},
   volume={255},
   date={2008},
   number={12},
   pages={3407--3430},
   issn={0022-1236},
   review={\MR{2469027}},
   doi={10.1016/j.jfa.2008.05.015},
}

\bib{G20}{article}{
   author={Garofalo, Nicola},
   title={On the best constant in the nonlocal isoperimetric inequality of
   Almgren and Lieb},
   journal={Atti Accad. Naz. Lincei Rend. Lincei Mat. Appl.},
   volume={31},
   date={2020},
   number={2},
   pages={465--470},
   issn={1120-6330},
   review={\MR{4120279}},
   doi={10.4171/rlm/900},
}

\bib{H20}{article}{
   author={Hardy, G. H.},
   title={Note on a theorem of Hilbert},
   journal={Math. Z.},
   volume={6},
   date={1920},
   number={3-4},
   pages={314--317},
   issn={0025-5874},
   review={\MR{1544414}},
   doi={10.1007/BF01199965},
}

\bib{H25}{article}{
   author={Hardy, G. H.},
   title={An inequality between integrals},
   journal={Messenger of Math.},
   volume={54},
   date={1925},
   pages={150--156},
}

\bib{L16}{article}{
   author={Larson, Simon},
   title={A bound for the perimeter of inner parallel bodies},
   journal={J. Funct. Anal.},
   volume={271},
   date={2016},
   number={3},
   pages={610--619},
   issn={0022-1236},
   review={\MR{3506959}},
   doi={10.1016/j.jfa.2016.02.022},
}

\bib{L20}{article}{
   author={Larson, Simon},
   title={Corrigendum to ``A bound for the perimeter of inner parallel
   bodies'' [J. Funct. Anal. 271 (3) (2016) 610--619]},
   journal={J. Funct. Anal.},
   volume={279},
   date={2020},
   number={5},
   pages={108574, 2},
   issn={0022-1236},
   review={\MR{4097280}},
   doi={10.1016/j.jfa.2020.108574},
}

\bib{L23}{book}{
   author={Leoni, Giovanni},
   title={A first course in fractional Sobolev spaces},
   series={Graduate Studies in Mathematics},
   volume={229},
   publisher={American Mathematical Society, Providence, RI},
   date={2023},
   pages={xv+586},
   isbn={[9781470468989]},
   isbn={[9781470472535]},
   isbn={[9781470472528]},
   review={\MR{4567945}},
   doi={10.1090/gsm/229},
}

\bib{LL01}{book}{
   author={Lieb, Elliott H.},
   author={Loss, Michael},
   title={Analysis},
   series={Graduate Studies in Mathematics},
   volume={14},
   edition={2},
   publisher={American Mathematical Society, Providence, RI},
   date={2001},
   pages={xxii+346},
   isbn={0-8218-2783-9},
   review={\MR{1817225}},
   doi={10.1090/gsm/014},
}

\bib{L19}{article}{
   author={Lombardini, Luca},
   title={Fractional perimeters from a fractal perspective},
   journal={Adv. Nonlinear Stud.},
   volume={19},
   date={2019},
   number={1},
   pages={165--196},
   issn={1536-1365},
   review={\MR{3912427}},
   doi={10.1515/ans-2018-2016},
}

\bib{LS10}{article}{
   author={Loss, Michael},
   author={Sloane, Craig},
   title={Hardy inequalities for fractional integrals on general domains},
   journal={J. Funct. Anal.},
   volume={259},
   date={2010},
   number={6},
   pages={1369--1379},
   issn={0022-1236},
   review={\MR{2659764}},
   doi={10.1016/j.jfa.2010.05.001},
}

\bib{M12}{book}{
   author={Maggi, Francesco},
   title={Sets of finite perimeter and geometric variational problems},
   series={Cambridge Studies in Advanced Mathematics},
   volume={135},
   publisher={Cambridge University Press, Cambridge},
   date={2012},
   pages={xx+454},
   isbn={978-1-107-02103-7},
   review={\MR{2976521}},
   doi={10.1017/CBO9781139108133},
}

\bib{MMP98}{article}{
   author={Marcus, Moshe},
   author={Mizel, Victor J.},
   author={Pinchover, Yehuda},
   title={On the best constant for Hardy's inequality in $\mathbf R^n$},
   journal={Trans. Amer. Math. Soc.},
   volume={350},
   date={1998},
   number={8},
   pages={3237--3255},
   issn={0002-9947},
   review={\MR{1458330}},
   doi={10.1090/S0002-9947-98-02122-9},
}

\bib{MS97}{article}{
   author={Matskewich, Tanya},
   author={Sobolevskii, Pavel E.},
   title={The best possible constant in generalized Hardy's inequality for
   convex domain in ${\mathbf R}^n$},
   journal={Nonlinear Anal.},
   volume={28},
   date={1997},
   number={9},
   pages={1601--1610},
   issn={0362-546X},
   review={\MR{1431208}},
   doi={10.1016/S0362-546X(96)00004-1},
}

\bib{M11}{book}{
   author={Maz'ya, Vladimir},
   title={Sobolev spaces with applications to elliptic partial differential
   equations},
   series={Grundlehren der mathematischen Wissenschaften [Fundamental
   Principles of Mathematical Sciences]},
   volume={342},
   edition={augmented edition},
   publisher={Springer, Heidelberg},
   date={2011},
   pages={xxviii+866},
   isbn={978-3-642-15563-5},
   review={\MR{2777530}},
   doi={10.1007/978-3-642-15564-2},
}

\bib{MS02}{article}{
   author={Maz'ya, V.},
   author={Shaposhnikova, T.},
   title={On the Bourgain, Brezis, and Mironescu theorem concerning limiting
   embeddings of fractional Sobolev spaces},
   journal={J. Funct. Anal.},
   volume={195},
   date={2002},
   number={2},
   pages={230--238},
   issn={0022-1236},
   review={\MR{1940355}},
   doi={10.1006/jfan.2002.3955},
}

\bib{MS03}{article}{
   author={Maz'ya, V.},
   author={Shaposhnikova, T.},
   title={Erratum to: ``On the Bourgain, Brezis and Mironescu theorem
   concerning limiting embeddings of fractional Sobolev spaces'' [J. Funct.
   Anal. {\bf 195} (2002), no. 2, 230--238; MR1940355 (2003j:46051)]},
   journal={J. Funct. Anal.},
   volume={201},
   date={2003},
   number={1},
   pages={298--300},
   issn={0022-1236},
   review={\MR{1986163}},
   doi={10.1016/S0022-1236(03)00002-8},
}

\bib{OK90}{book}{
   author={Opic, B.},
   author={Kufner, A.},
   title={Hardy-type inequalities},
   series={Pitman Research Notes in Mathematics Series},
   volume={219},
   publisher={Longman Scientific \& Technical, Harlow},
   date={1990},
   pages={xii+333},
   isbn={0-582-05198-3},
   review={\MR{1069756}},
}

\bib{S23}{article}{
   author={Sk, Firoj},
   title={Characterization of fractional Sobolev-Poincar\'e and (localized) Hardy inequalities},
   journal={J. Geom. Anal.},
   volume={33},
   date={2023},
   number={7},
   pages={Paper No. 223, 20},
   issn={1050-6926},
   review={\MR{4581748}},
   doi={10.1007/s12220-023-01293-y},
}

\bib{V91}{article}{
   author={Visintin, Augusto},
   title={Generalized coarea formula and fractal sets},
   journal={Japan J. Indust. Appl. Math.},
   volume={8},
   date={1991},
   number={2},
   pages={175--201},
   issn={0916-7005},
   review={\MR{1111612}},
   doi={10.1007/BF03167679},
}

\end{biblist}
\end{bibdiv}
\end{document}